\documentclass{ejm}

\usepackage{amssymb}
\usepackage{amsfonts}
\usepackage{mathtools}
\DeclareMathAlphabet\mathbfcal{OMS}{cmsy}{b}{n}
\usepackage{amssymb}
\usepackage{tikz}
\usetikzlibrary{trees}
\usepackage{graphicx}
\usepackage{wrapfig}
\usepackage{booktabs,multirow}
\usepackage{hhline}
\usepackage[utf8]{inputenc}
\usepackage[english]{babel}

\usepackage[utf8]{inputenc}
\usepackage[T1]{fontenc}
\usepackage{textcomp}
\usepackage{siunitx}
\usepackage{adjustbox}

\usepackage{lipsum}
\usepackage{color}
\usepackage{gensymb}
\usepackage{epstopdf}
\usepackage{bm}
\usepackage[english]{babel}
\usepackage{placeins}
\usepackage{makecell}
\usepackage{tabularx}
\usepackage{varioref}
\usepackage{rotating}
\usepackage{subcaption}
\usepackage{mathrsfs}
\usepackage{float}
\usepackage{setspace}
\usepackage[utf8]{inputenc}
\usepackage{fancyhdr}
\usepackage{graphicx}
\usepackage{adjustbox,lipsum}
\usepackage{hyperref}
\usepackage{multirow}
\usepackage{tabularx}
\usepackage{enumerate}
\usepackage{enumitem}
\usepackage[utf8]{inputenc}
\usepackage[T1]{fontenc}
\usepackage[utf8]{inputenc}
\usepackage[english]{babel}

\let\realverbatim=\verbatim
\let\realendverbatim=\endverbatim
\renewcommand\verbatim{\par\addvspace{6pt plus 2pt minus 1pt}\realverbatim}
\renewcommand\endverbatim{\realendverbatim\addvspace{6pt plus 2pt minus 1pt}}

\ifprodtf \else
  \checkfont{eurm10}
  \iffontfound
    \IfFileExists{upmath.sty}
      {\typeout{^^JFound AMS Euler Roman fonts on the system,
                   using the 'upmath' package.^^J}%
       \usepackage{upmath}}
      {\typeout{^^JFound AMS Euler Roman fonts on the system, but you
                   don't seem to have the}%
       \typeout{'upmath' package installed. EJM.cls can take advantage
                 of these fonts,^^Jif you use 'upmath' package.^^J}%
      }
  \else
  \fi
\fi

\ifprodtf \else
  \checkfont{msam10}
  \iffontfound
    \IfFileExists{amssymb.sty}
      {\typeout{^^JFound AMS Symbol fonts on the system, using the
                'amssymb' package.^^J}%
       \usepackage{amssymb}%
         \let\leq=\leqslant
         \let\geq=\geqslant
      }{}
  \fi
\fi

\ifprodtf \else
  \IfFileExists{amsbsy.sty}
    {\typeout{^^JFound the 'amsbsy' package on the system, using it.^^J}%
     \usepackage{amsbsy}}
    {}
\fi

\newsavebox{\astrutbox}
\sbox{\astrutbox}{\rule[-5pt]{0pt}{20pt}}

\usepackage{setspace}
\usepackage{amsmath}
\usepackage{bm}

\newtheorem{theorem}{Theorem}[section]
\newdefinition{definition}[theorem]{Definition}

\title[European Journal of Applied Mathematics]{Domain-dependent stability analysis and parameter classification of a reaction-diffusion model on spherical geometries}

\author[W. Sarfaraz et al.]{%
  W.\ns  S\ls A\ls R\ls F\ls A\ls R\ls A\ls Z$\,^1$,\ns
  \and
  A.\ns M\ls A\ls D\ls Z\ls V\ls A\ls M\ls U\ls S\ls E$\,^2$
}

\affiliation{%
  $^1 \,$University of Sussex, School of Mathematical and Physical Sciences, Department of Mathematics, Pevensey 3, Brighton, BN1 9QH, UK\\
  $^1 \,$email\textup{\nocorr: \texttt{wakilsarfaraz@gmail.com}}\\
  $^2\,$University of Sussex, School of Mathematical and Physical Sciences, Department of Mathematics, Pevensey 3, Brighton, BN1 9QH, UK\\
$^2 \,$email\textup{\nocorr: \texttt{a.madzvamuse@sussex.ac.uk}}\\}

\date{09 January 2018}
\pubyear{2000}
\volume{000}
\pagerange{\pageref{firstpage}--\pageref{lastpage}}

\begin{document}

\label{firstpage}
\maketitle

\begin{abstract}
In this work an \textit{activator-depleted} reaction-diffusion system is investigated on polar coordinates with the aim of exploring the relationship and the corresponding influence of domain size on the types of possible diffusion-driven instabilities. Quantitative relationships are found in the form of necessary conditions on the area of a disk-shape domain with respect to the diffusion and reaction rates for certain types of diffusion-driven instabilities to occur. Robust analytical methods are applied to find explicit expressions for the eigenvalues and eigenfunctions of the diffusion operator on a disk-shape domain with homogenous Neumann boundary conditions in polar coordinates. Spectral methods are applied using chebyshev non-periodic grid for the radial variable and Fourier periodic grid on the angular variable to verify the nodal lines and eigen-surfaces subject to the proposed analytical findings. 
The full classification of the parameter space in light of the bifurcation analysis is obtained and numerically verified by finding the solutions of the partitioning curves inducing such a classification. Furthermore, analytical results are found relating the area of (disk-shape) domain with reaction-diffusion rates in the form of necessary conditions for the different types of bifurcations. These results are on one hand presented in the form of mathematical theorems with rigorous proofs, and, on the other hand using finite element method, each claim of the corresponding theorems are verified by obtaining the theoretically predicted behaviour of the dynamics in the numerical simulations. Spatio-temporal periodic behaviour is demonstrated in the numerical solutions of the system for a proposed choice of parameters and a rigorous proof of the existence of infinitely many such points in the parameter plane is presented under a restriction on the area of the domain, with a lower bound in terms of reaction-diffusion rates.   
\end{abstract}

\begin{keywords}
Reaction-diffusion systems, Dynamical systems, Bifurcation analysis, Stability analysis, Turing diffusion-driven instability , Hopf Bifurcation, Transcritical bifurcation, Parameter spaces, Polar coordinates, Curved boundary
\end{keywords}

\section{Introduction}
\label{sec1}
Analysis of reaction-diffusion systems (RDSs) in the context of pattern formation is a widely studied topic \cite{book1, book7, paper1, paper2, paper23, paper30, paper31, paper32} in many branches of scientific research. Scholars of mathematical and computational biology \cite{paper3,paper4,paper5, paper6, paper7, paper11, paper13, paper17} devoted a great deal of attention to fully explore the theory of the dynamics governed by reaction-diffusion systems for a variety of reaction kinetics. Few of the routinely used and popular reaction kinetics in the theory of reaction-diffusion systems are \textit{activator-depleted} \cite{paper29, paper28, paper26, paper27, paper22, paper20, thesis2}, Gierer-Meinhardt \cite{paper24, paper23} and Thomas reaction kinetics \cite{paper25, paper37}. It is crucial to realise that the ultimate and complete knowledge encapsulating all aspects of reaction-diffusion systems for all types of possible reaction-kinetics is beyond the scope and feasibility of a single research paper or even a single book. Therefore, literature to date \cite{paper37}-\cite{book1} on the analysis of reaction-diffusion systems falls under a natural classification of approaches based on the application of methods corresponding to the speciality of a particular scholar. It is in this spirit that the current work is a natural extension of \cite{paper38}, in which a full classification of the admissible parameter space associated to the dynamical behavior of \textit{activator-depleted} reaction-diffusion system was explored on stationary rectangular domains. It is a common hypothesis in the theory of biological pattern formation \cite{paper39, paper13}, that the emergence of spatially periodic pattern in the computational simulations of reaction-diffusion systems is connected to the properties of the associated eigenfunctions of the diffusion operator in the corresponding domain. The associated drawback with computational approaches \cite{paper10, paper12, paper15, paper16, paper19, paper21, paper37, book4, book5, book6} for reaction-diffusion systems is that it lacks to provide rigorous insight of the role of the eigenfunctions in the emergence of spatial pattern, thus forming a natural platform to explore the evolution of spatial patterns from a perspective of dynamical systems. Numerous researchers have contributed to exploring the computational aspect of reaction-diffusion systems \cite{paper4, paper12, paper14, paper15, paper26, paper27}, in which a restricted choice of parameter values are used with a partial reliance on trial and error to obtain the emergence of an evolving pattern either in space or in time. A large variety of scholars have also employed the analytical approach \cite{paper1, paper9, paper13, paper17, paper18, paper20, paper28}, to explore reaction-diffusion systems from a perspective of dynamical systems. Majority of cases with analytical approaches concluded with certain conditions \cite{paper9, paper17, paper18, paper20, paper28, paper31} on the characteristics of the stability matrix that theoretically predicts reaction-diffusion systems to exhibit spatial patterns, lacking to extend the analysis to computational consequences of these conditions on the admissible parameter spaces in terms of diffusion-driven instability. Furthermore, from the literature to date \cite{paper28, paper30, paper39, paper45, paper46, paper47, paper48}, it is evident that insufficient attention is given to rigorously explore the influence of reaction-diffusion rates on the dynamical behaviour of such systems in the context of domain size. The motivation of the current work originates from the hypothesis of the spatial dependence of the eigenfunctions of the diffusion operator with the domain size itself and extending this idea to further investigate how this spatial dependence induces an influence on the linearised stability-matrix of a reaction-diffusion system through the properties of the eigenfunctions of the diffusion operator.  With an attempt to further contribute to the current knowledge of reaction-diffusion systems, the present paper employs a set of rigorous findings obtained from bifurcation analysis of the \textit{activator-depleted} reaction-diffusion model to explore the corresponding influence induced on the admissible parameter spaces and diffusion-driven instability. Despite the restriction of the current work to a particular type of reaction kinetics namely \textit{activator-depleted}, the methodology can however serve as a framework for developing a standard independent method by combining the known aspects of exploring the topic namely bifurcation analysis, parameter spaces and computational methods, that could be utilised for general reaction-diffusion systems. The lack of a self-contained complete methodology that combines all the known distinct aspects of exploring this topic creates a credible argument for the importance of this work. The contents of the current paper offer a natural extension of the work presented in \cite{paper38}, where the domain of solution was restricted to rectangular geometries. Equivalent results to those presented in \cite{paper38} are found in the present work, except that the domain of solution in this work is a two dimensional disk-shape geometry bounded by a circle. The bulk of the current work consists of rigorously proven statements that quantitatively relate the radius of a disk-shape solution domain to reaction-diffusion rates in light of bifurcation analysis, which are computationally verified using the finite element method.

The contents of this work are structured such that in Section \ref{absenceone} the model equations for \textit{activator-depleted} reaction-diffusion system are stated in cartesian coordinates and transformed to polar coordinates in its non-dimensional form with the corresponding initial and boundary conditions.  Section \ref{linearisation} consists of a detailed analytical method to explicitly derive eigenfunctions and the corresponding eigenvalues satisfying the boundary conditions prescribed for \textit{activator-depleted} reaction-diffusion system in Section \ref{absenceone}. Furthermore, in Section \ref{linearisation}, an application of spectral approach on polar coordinates using chebyshev and Fourier grid method is presented \cite{book8} to demonstrate and verify the analytical derivation of the eigenfunctions as well as the stability matrix and the corresponding characteristic polynomial for a linearised approximation of the original system. Section \ref{bifur} presents analytical results on the relationship of domain-size (radius of disk) with different types of bifurcations in the dynamics. Section \ref{main} is devoted to the parameter space classification in light of bifurcation analysis of the uniform steady-state of the system. A numerical method for computing the partitioning curves for such classification is presented in detail. The shift of parameter spaces as a consequence of changing reaction-diffusion rates is investigated and theorems are proven that rigorously establish the relation of domain size (radius of disk) with the corresponding types of diffusion-driven instabilities. Furthermore, a full classification of the admissible parameter space in terms of stability and types of the uniform steady state of the system is also presented in Section \ref{bifur}, where it is shown that if the radius of a disk-shape domain satisfies certain inequalities in terms of reaction-diffusion rates, then the admissible parameter space allows or forbids certain types of bifurcations in the dynamics of the reaction-diffusion system. Section \ref{fem} presents numerical simulations of an \textit{activator-depleted} reaction-diffusion system using finite element method to verify the proposed classification of parameter spaces and the theoretically predicted behaviour in the dynamics. Due to the curved boundary of the domain a non-standard technique called \textit{distmesh} \cite{paper40, paper41, paper44} is used to obtain discretisation for simulating the finite element method. The technicality of the algorithm for \textit{distmesh} is briefly explained. Section \ref{conclusion} presents the conclusion and possible extensions of the current work.

\section{Model equations}\label{absenceone}
Two chemical species $u$ and $v$ are modelled by the well-known \textit{activator-depleted} reaction-diffusion system, where both species are coupled through nonlinear reaction terms. The system assumes independent diffusion rates for both of the species. The RDS is considered on a two dimensional circular domain denoted by $\Omega \in \mathbb{R}^2$, where $\Omega$ is defined by 
\[
\Omega = \{(x,y)\in \mathbb{R}^2: x^2+y^2<\rho^2\}.
\]
This forms a two dimensional disk-shape domain with a circle forming its boundary denoted by $\partial \Omega$, which contains the points given by 
\[
\partial\Omega = \{(x,y)\in \mathbb{R}^2: x^2+y^2 = \rho^2\}.
\]
The RDS with \textit{activator-depleted} reaction kinetics in its non-dimensional form on cartesian coordinates has the form
\begin{equation}
\begin{cases}
\begin{cases}
\frac{\partial u}{\partial t} =& \triangle u + \gamma (\alpha -u+u^2v),  \qquad (x,y)\in\Omega,\quad t>0,\\
\frac{\partial v}{\partial t} =& d \triangle v + \gamma( \beta - u^2v),
\end{cases} \\
\frac{\partial u}{\partial \bm{n}} = \frac{\partial v}{\partial \bm{n}}=0, \qquad \qquad \qquad \text{on} \quad (x,y) \in \partial \Omega, \quad t\geq 0,\\
u(x,y,0) = u_0(x,y), \qquad v(x,y,0) = v_0(x,y),\qquad(x,y) \in \Omega, \quad t=0,     
\end{cases}
\label{A}
\end{equation}
where $\alpha$, $\beta$, $\gamma$ and $d$ are strictly positive real constants. In system (\ref{A}) the positive parameter $d$ denotes the non-dimensional ratio of diffusion rates given by $d=\frac{D_v}{D_u}$, where $D_v$ and $D_u$ are the independent diffusion rates of $v$ and $u$ respectively. The non-dimensional parameter $\gamma$ is known as the scaling parameter of (\ref{A}), which quantifies the reaction rate. The boundary of $\Omega$ under the current study assumes homogeneous Neumann boundary conditions (also known as zero flux boundary conditions), which means that the chemical species $u$ and $v$ can neither escape nor enter through $\partial \Omega$. 
$\bm{n}$ denotes the normal to $\partial \Omega$ in the outward direction. Initial conditions for (\ref{A}) are prescribed 
as positive bounded continuous functions $u_0(x,y)$ and $v_0(x,y)$.
System (\ref{A}) can be spatially transformed to polar coordinates using $x = r \cos \theta$ and $y = r \sin \theta$ to obtain
\begin{equation}
\begin{cases}
\begin{cases}
   \frac{\partial u}{\partial t} =& \triangle_p u + \gamma f(u,v), \\
   \frac{\partial v}{\partial t} =& d \triangle_p v + \gamma  g(u,v),
\end{cases} \\
\frac{\partial u}{\partial r}\big |_{r=\rho} = \frac{\partial v}{\partial r}\big |_{r=\rho}=0, \qquad (r,\theta)\in\partial\Omega, \quad t\geq 0,\\
u(r,\theta,0)=u_0(r,\theta), \qquad v(r,\theta,0)=v_0(r,\theta), \qquad (r,\theta) \in \Omega, \quad t=0, 
      \end{cases}
   \label{polarsystem}
\end{equation}
where $\triangle_p$ denotes the Laplace operator in polar coordinates written as 
\begin{equation}
\triangle_p u(r,\theta)=\frac{1}{r}\frac{\partial}{\partial r}\Big(r\frac{\partial u}{\partial r}\Big)+\frac{1}{r^2}\frac{\partial^2u}{\partial \theta^2}.
\label{polaplace}
\end{equation} 
In (\ref{polarsystem}),  $u$ and $v$ depend on coordinates $(r,\theta)$ and the functions $f$ and $g$ are defined by $f(u,v)=\alpha-u+u^2v$ and $g(u,v)=\beta-u^2v$. Initial and boundary conditions are transformed in a similar fashion. In (\ref{polarsystem}) the strictly positive constants namely $\alpha$, $\beta$, $d$ and $\gamma$ remain to satisfy exactly the same definitions as in (\ref{A}). 

\section{Stability analysis of the reaction-diffusion system}\label{linearisation}
\subsection{Stability analysis of the reaction-diffusion system in the presence of diffusion}
Let $u_s$ and $v_s$ denote the uniform steady state solution satisfying system  (\ref{A}) (and equivalently system (\ref{polarsystem})) with  \textit{activator-depleted} reaction kinetics in the absence of diffusion and these are given by  $(u_s,v_s)=(\alpha+\beta, \frac{\beta}{(\alpha+\beta)^2})$ \cite{book1, book7, paper38, paper17}. 
For linear stability analysis, system (\ref{polarsystem}) is perturbed in the neighbourhood of the uniform steady state $(u_s,v_s)$ and the dynamics of the perturbed system are explored i.e. $(u,v)=(\bar{u}+u_s, \bar{v}+v_s)$, where $\bar{u}$ and $\bar{v}$ are assumed small. In system (\ref{polarsystem}) the variables $u$ and $v$ are substituted by the expression in terms of $(\bar{u},\bar{v})$ and $(u_s,v_s)$ and expanded using Taylor expansion for functions of two variables up to and including the linear terms with the higher order terms discarded. This leads to the linearised version of system (\ref{polarsystem}) written in matrix form as
\begin{equation}
\frac{\partial}{\partial t}\left[\begin{array}{c}
     \bar{u}  \\
     \bar{v} 
\end{array}\right]=\left[\begin{array}{cc}
     1&0  \\
     0&d 
\end{array}\right]\left[\begin{array}{c}
     \triangle_p \bar{u} \\
     \triangle_p \bar{v} 
\end{array}\right]+\left[\begin{array}{cc}
     \frac{\partial f}{\partial u}(u_s,v_s)& \frac{\partial f}{\partial v}(u_s,v_s)  \\
     \frac{\partial g}{\partial u}(u_s,v_s)&\frac{\partial g}{\partial u}(u_s,v_s)
\end{array}\right]\left[\begin{array}{c}
     \bar{u} \\
     \bar{v} 
\end{array}\right].
\label{vect}
\end{equation}
The next step to complete the linearisation of system (\ref{polarsystem}) is to compute the eigenfunctions of the diffusion operator namely $\triangle_p$, which will require to find the solution to an elliptic 2 dimensional eigenvalue problem on a disk that satisfies homogeneous Neumann boundary conditions prescribed for system (\ref{polarsystem}).

The solution of eigenvalue problems on spherical domains is a well studied area \cite{book9, paper33, book10}, with the majority of research focused on problems with boundary-free manifolds such as circle, torus and/or sphere. Considering the restriction imposed from the boundary conditions prescribed for the current problem, entails that the case requires explicit detailed treatment to rigorously find the eigenfunctions satisfying these boundary conditions. Therefore, it is important to present a step-by-step demonstration of the process, starting with writing out the relevant eigenvalue problem all the way through to finding the closed form solution in the form of an infinite set of eigenfunctions satisfying such an eigenvalue problem. The eigenvalue problem we want to solve is of the form 
\begin{equation}    \label{eigen}
\begin{cases}
    \triangle_p w = -\eta^2 w, \qquad \eta \in \mathbb{R},
\\
\frac{\partial w}{\partial r}\big |_{r=\rho}=0, \qquad \rho\in\mathbb{R}_+\backslash\{0\},
\end{cases}
\end{equation}
where $\triangle_p$ is the diffusion operator in polar coordinates defined by (\ref{polaplace}), on a disk with radius $\rho$. We use separation of variables to obtain the solutions of (\ref{eigen}). Application of separation of variables to problem (\ref{eigen}) requires a solution of the form $w(r,\theta)=R(r)\Theta(\theta)$, which is substituted into problem (\ref{eigen}) to obtain 
\begin{equation}
\frac{d^2R}{dr^2}\Theta+\frac{1}{r}\frac{dR}{dr}\Theta+\frac{1}{r^2}R\frac{d^2\Theta}{d\theta}=-\eta^2R\Theta.
\label{sep1}
\end{equation}
Dividing both sides of (\ref{sep1}) by $R\Theta$, multiplying throughout by $r^2$ and rearranging, yields  
\begin{equation}
r^2\frac{1}{R}\frac{d^2R}{dr^2}+r\frac{1}{R}\frac{dR}{dr}+\eta^2r^2=-\frac{1}{\Theta}\frac{d^2\Theta}{d\theta^2}.
\label{sep2}
\end{equation} 
Using anzats of the form $\Theta(\theta)=\exp(in\theta)$ and a change of variable $x = \eta r$ is applied to (\ref{sep2}). Taking the usual steps of Frobenius method it can be shown that the general solution $R(x)$ satisfying (\ref{sep2}) takes the form
$R(x)=R^1(x)+R^2(x)$, where $R^1(x)$ and $R^2(x)$ are respectively given by 
\begin{equation}
R^1(x)=x^n\sum_{j=0}^{\infty}\frac{(-1)^jC_0x^{2j}}{4^j\times j!\times(n+j)\times(n+j-1)\times \cdot\cdot\cdot\times(n+1)}
\label{sol1}
\end{equation}
and 
\begin{equation}
R^2(x)=x^{-n}\sum_{j=0}^{\infty}\frac{(-1)^jC_0x^{2j}}{4^j\times j!\times(-n+j)\times(-n+j-1)\times \cdot\cdot\cdot\times(-n+1)}.
\label{sol2}
\end{equation}
Series solutions (\ref{sol1}) and (\ref{sol2}) are referred to as the Bessel functions of the first kind \cite{book11, book10}. Before writing the general solution to problem (\ref{eigen}), it must be noted that we require $R$ to be a function of $r$ and not of $x$, bearing in mind that $r$ is related to $x$ under the linear transformation given by $x=\eta r$. Therefore, the general solution to problem (\ref{eigen}) can be written in the form 
\begin{equation}
w(r,\theta)=R(x(r))\Theta(\theta),
\label{eigenf}
\end{equation}
where $R(x(r))=R_1(x(r))+R_2(x(r))$ and $\Theta(\theta)=\exp(in\theta)$, with $R_1$ and $R_2$ defined by (\ref{sol1}) and (\ref{sol2}) respectively.
The homogeneous Neumann boundary conditions are imposed on the set of eigenfunctions (\ref{eigenf}), and noting that the flux is independent of the variable $\theta$, we have 
$\frac{\partial w}{\partial r}\big|_{r=\rho}=0$ which implies that   $\frac{dR}{dr}\big|_{r=\rho} = 0$.
A straightforward application of chain rule yields  
$\frac{dR}{dr}\big|_{r=\rho}=\eta\frac{dR}{dx}\big|_{x=\eta \rho}$.

Let $a_j$ and $b_j$ denote the coefficients corresponding the $jth$ term in the infinite series for $R_1(x)$ and $R_2(x)$ respectively, then for every $j \in \mathbb{N}$ the expressions for $a_j$ and $b_j$ take the forms
\begin{equation}
a_j = \frac{(-1)^jC_0}{4^j\times j!\times(n+j)\times(n+j-1)\times \cdot\cdot\cdot\times(n+1)},
\label{coeff1}
\end{equation}
\begin{equation}
b_j = \frac{(-1)^jC_0}{4^j\times j!\times(-n+j)\times(-n+j-1)\times \cdot\cdot\cdot\times(-n+1)}.
\label{coeff2}
\end{equation}
This entails that $R(x)$ can be written in the form 
$R(x)=\sum_{j=0}^{\infty}\big(a_jx^{n+2j}+b_jx^{-n+2j}\big)$.
Differentiating $R(x)$ with respect to $x$ and equating it to zero, on substituting $x=\eta \rho$ and using the chain rule we obtain the equation
\begin{equation}\begin{split}
0 = & \frac{dR}{dr}\Big|_{r=\rho}=\eta\sum_{j=0}^{\infty}\big[a_j(n+2j)x^{n+2j-1}+b_j(-n+2j)x^{-n+2j-1}\big]\Big|_{x=\eta\rho}\\ 
=&\eta\sum_{j=0}^{\infty}\big[a_j(n+2j)(\eta\rho)^{n+2j-1}+b_j(-n+2j)(\eta \rho)^{-n+2j-1}\big]\\
=&\eta\Big[(\eta\rho)^n\sum_{j=0}^{\infty}a_j(n+2j)(\eta\rho)^{2j-1}+(\eta\rho)^{-n}\sum_{j=0}^{\infty}b_j(-n+2j)(\eta \rho)^{2j-1}\Big],
\label{coeff4}\end{split}\end{equation}
which holds true if and only if both of the summations in (\ref{coeff4}) are independently zero. Investigating the first summation in (\ref{coeff4}) and expanding it for a few successive terms, it can be shown that the successive terms carry alternating signs (due to the expression for $a_j$), therefore, the only way the infinite series can become zero is if each of the successive terms cancel one another. Let $F_j$ and $S_j$ denote the $jth$ terms of the first and second summations in (\ref{coeff4}) respectively, then for (\ref{coeff4}) to hold true, $F_j+F_{j+1}=0$ and $S_j+S_{j+1}=0$ must be true for all $j\in\mathbb{N}$. The full expressions for $F_j$ and $F_{j+1}$ can be written as
\begin{equation}
F_j=\frac{(\eta\rho)^n(-1)^jC_0(n+2j)(\eta\rho)^{2j-1}}{4^j\times j!\times(n+j)\times(n+j-1)\times \cdot\cdot\cdot\times(n+1)}
\label{fj}
\end{equation}
and 
\begin{equation}
F_{j+1}=\frac{(\eta\rho)^n(-1)^{j+1}C_0(n+2j+2)(\eta\rho)^{2j+1}}{4^{j+1}\times (j+1)!\times(n+j+1)\times(n+j)\times \cdot\cdot\cdot\times(n+1)}.
\label{fj+1}
\end{equation}
For the first $(j=0,j=1)$, second $(j=2,j=3)$ and third $(j=4,j=5)$ successive pairs, independently equating the sum of the expressions (\ref{fj}) and (\ref{fj+1}) to zero yields that $\eta^2_{n,k}$ can be written as 
\begin{equation}
\eta_{n,k}^2=\frac{4(2k+1)(n+2k+1)(n+4k)}{\rho^2(n+4k+2)}.
\label{eigenvalue}
\end{equation}
\subsubsection{\bf Remark}\label{remark3}
{\it The restriction on the order of the corresponding Bessel's equation namely $n\in\mathbb{R}\backslash\frac{1}{2}\mathbb{Z}$ in Theorem \ref{theorem1} can be relaxed by employing Bessel's function of the second kind \cite{book9, book10, book11} and imposing on it the homogeneous Neumann boundary conditions. Therefore, for the purpose of the current study the order of the corresponding Bessel's equation is precluded from becoming a full or half integer. Bear in mind that the proposed choice of $n\in\mathbb{R}\backslash\frac{1}{2}\mathbb{Z}$ makes the set of eigenvalues $\eta_{n,k}^2$ a semi discrete infinite set. The spectrum is discrete with respect to positive integers $k$ and continuous (uncountable) with respect to $n$. }

\begin{theorem}\label{theorem1}
 Let $w(r,\theta)$ satisfy problem (\ref{eigen}) with homogeneous Neumann boundary conditions. Given that the order $n$ of the associated Bessel's function belongs to the set $\mathbb{R}\backslash\frac{1}{2}\mathbb{Z}$, [see Remark \ref{remark3}] then for a fixed $n$ there exists an infinite set of eigenfunctions of the diffusion operator $\triangle_p$ as defined in (\ref{polaplace}), which is given by 
 \begin{equation}
  w_{n,k}=\big[R^1_{n,k}(r)+R^2_{n,k}(r)\big]\Theta_n(\theta)
  \label{eigenfunc}
 \end{equation}
with the explicit expressions for $R_{n,k}^1(r)$, $R_{n,k}^2(r)$ and $\Theta_n(\theta)$ as 
\begin{equation}
R^1_{n,k}(r)=\sum_{j=0}^{\infty}\frac{(-1)^jC_0(\eta_k r)^{2j+n}}{4^j j!(n+j)(n+j-1) \cdot\cdot\cdot(n+1)},
\end{equation}
\begin{equation}
R^2_{n,k}(r)=\sum_{j=0}^{\infty}\frac{(-1)^jC_0(\eta_k r)^{2j-n}}{4^j j!(-n+j)(-n+j-1) \cdot\cdot\cdot(-n+1)} 
\end{equation}
and
\begin{equation}
\Theta_n(\theta)= \exp(in\theta),
\end{equation}
where for the $kth$ successive pair $\eta_{n,k}$ satisfies (\ref{eigenvalue}), when ever $j=2k$. 
\begin{proof}
The proof consists of all the steps from (\ref{sep1}) to (\ref{eigenvalue}).
\end{proof}
\end{theorem}
It is worth noting that for a set containing $j=2k$ terms from the series solutions defining the eigenfunctions (\ref{eigenfunc}), one can only obtain $k$ eigenvalues that satisfy problem (\ref{eigen}). Therefore, it is important to realise that unlike the eigenfunctions of a diffusion operator \cite{paper38, thesis2} on a rectangular geometry, here only pairwise terms can qualify in the series solutions to be claimed as the eigenfunction satisfying problem (\ref{eigen}) and not individual terms associated discretely to the summation index $j$.
\begin{figure}[ht]
 \centering
 \small
  \begin{subfigure}[h]{.495\textwidth}
    \centering
 \includegraphics[width=\textwidth]{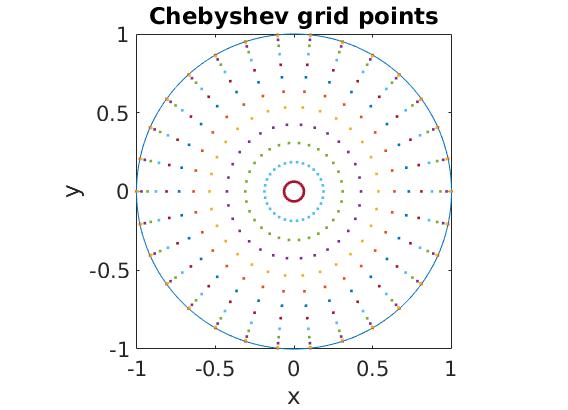}
 \caption{Coarse mesh structure of chebyshev radial\\ grid with $N=25$ and using the periodic \\Fourier grid with $M=30$, leading to \\an angular step-size of $12\si{\degree}$.}
 \label{paramfig7}
 \end{subfigure}
 \begin{subfigure}[h]{.495\textwidth}
 \centering
    \includegraphics[width=\textwidth]{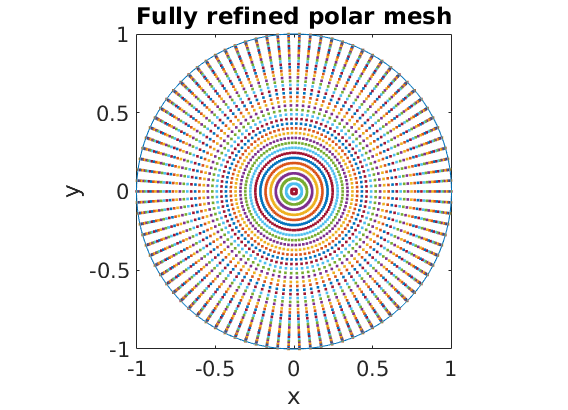}
  \caption{Fully refined polar mesh with $N=95$ chebyshev grid-points and $M=90$ Fourier periodic grid-points, with angular step-size of $4\si{\degree}$.}
  \label{paramfig9}
 \end{subfigure}
 \caption{Mesh generation on polar coordinates obtained from the combination of the chebyshev grid on radial axis and the Fourier periodic grid on angular axis.}
  \label{figmesh}
\end{figure}\subsection{Numerical experiments using the spectral method}
Let $R_{n,k}(r)$ denote the expression $R^1_{n,k}(r)+R^2_{n,k}(r)$, then the full set of eigenfunctions to problem (\ref{eigen}), can be written as $w(r,\theta)=\sum_{k=0}^{\infty}R_{n,k}(r)\Theta_n(\theta)$. For numerical demonstration of Theorem \ref{theorem1}, the spectral method \cite{book8} is employed on a unit disk centred at the origin of the $x,y$ coordinates. A spectral mesh in polar coordinates is constructed on a unit disk $\Omega=\{(r,\theta)\in\mathbb{R}^2: r\in [0,1], \theta\in[0,2\pi]\}$, where a periodic Fourier grid is applied to the angular axis $\theta$ and a non-periodic chebyshev grid is applied to the radial axis $r$ \cite{book8}. In order to tackle the singularity at $r=0$, chebyshev discretisation is applied to the whole of the diameter of $\Omega$, which means that $-1 \leq r \leq 1$ is used instead of $0\leq r \leq 1$. The interval $[-1, 1]$ is discretised in the form $r_i=\cos(\frac{i\pi}{N})$, for $i=0,1,2,...,N$, where $N=2j+1,j\in \mathbb{N}$, is a positive odd integer. The odd number of points on the chebyshev grid serves to preclude complications that could rise from the singular point $r=0$ by locating it radially in the middle of two successive chebyshev grid points. The second coordinate $\theta$ is discretised using $\theta_i=\frac{2i\pi}{M}$, for $i=0,1,2,...,M$, where $M=2j, j\in \mathbb{N}$, is a positive even integer. For the full details on the implementation of the spectral method in polar coordinates the reader is referred to \cite{book8}. Figure \ref{figmesh} (a) shows a coarse mesh structure on a unit disk using $N=25$ and $M=30$, with angular step-size of $12\si{\degree}$. Similarly, Figure \ref{figmesh} (b) shows the refined mesh using $N=95$ and $M=90$ with angular step-size of $4\si{\degree}$, on which all the simulations for eigenfunctions $w(r,\theta)$ are performed to visualise the corresponding nodal lines and surfaces.
Colour encoded plots, corresponding to nine different modes $k=\{1,2,3,5,6,7,9,12,15\}$ are simulated on the mesh given in Figure \ref{figmesh} (b) using a technique presented in \cite{paper35, paper36}. Due to the fact that eigenfunctions $w(r,\theta)$ are complex valued functions, therefore, trivial methods of visualising a real valued function of two variables do not suffice to give a meaningful representation to a complex valued function. Note that in the expression for the eigenfunctions $w(r,\theta)=R(r)\Theta(\theta)$, only the function $\Theta(\theta)$ contains the non-zero imaginary part, therefore, the variable $\theta$ is encoded by a polar colour scheme (Hue, Saturation Value) and displayed directly on $\Omega$. For full details on the implementation of this process the reader is referred to \cite{paper35, paper36}. Figure \ref{eigenfunction} shows a colour encoded representation of nine different modes, each of which corresponds to one of the nine nodal line depictions of $w_{n,k}(r,\theta)$ in Figure \ref{nodal}. 
\begin{figure}[!ht]
 \centering
 \small
  \begin{subfigure}[h]{0.3\textwidth}
 \centering
    \includegraphics[width=\textwidth]{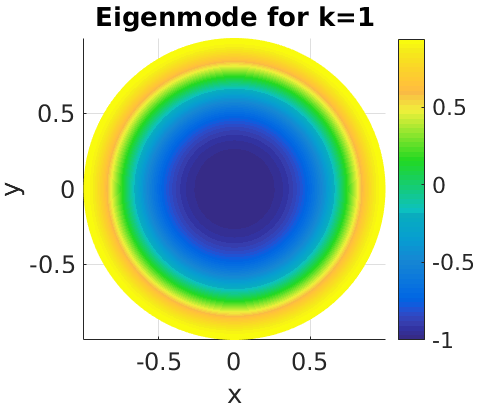}
  \caption{}
  \label{kone}
 \end{subfigure}
  \begin{subfigure}[h]{0.3\textwidth}
 \centering
 \includegraphics[width=\textwidth]{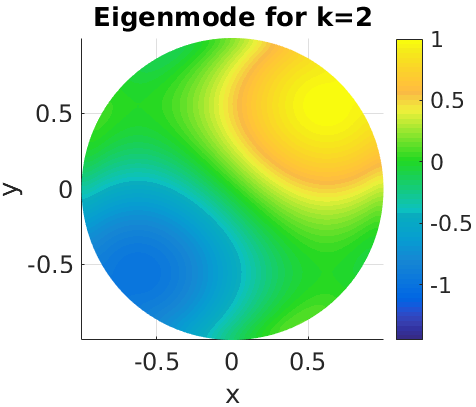}
  \caption{}
  \label{ktwo}
 \end{subfigure}
  \begin{subfigure}[h]{0.3\textwidth}
 \centering
    \includegraphics[width=\textwidth]{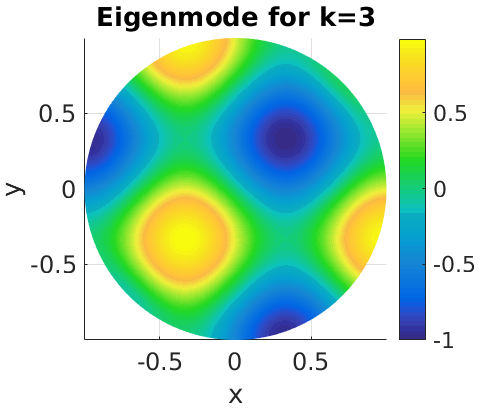}
  \caption{}
  \label{kthree}
 \end{subfigure}
 \begin{subfigure}[h]{0.320\textwidth}
 \centering
    \includegraphics[width=\textwidth]{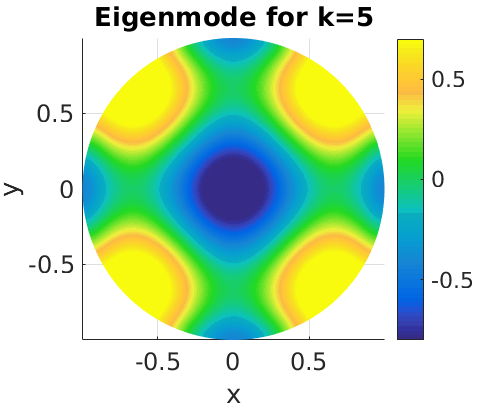}
  \caption{}
  \label{kfive}
 \end{subfigure}
  \begin{subfigure}[h]{0.3\textwidth}
 \centering
 \includegraphics[width=\textwidth]{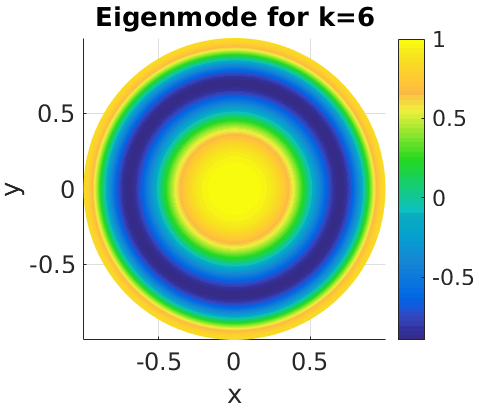}
  \caption{}
  \label{ksix}
 \end{subfigure}
  \begin{subfigure}[h]{0.3\textwidth}
 \centering
    \includegraphics[width=\textwidth]{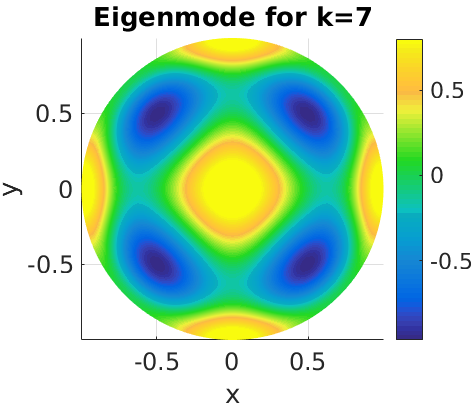}
  \caption{}
  \label{kseven}
 \end{subfigure}
 \begin{subfigure}[h]{0.3\textwidth}
 \centering
    \includegraphics[width=\textwidth]{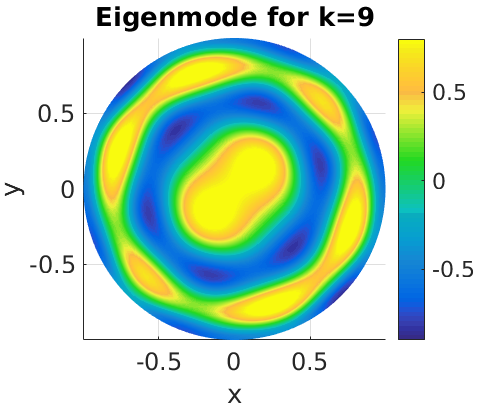}
  \caption{}
  \label{knine}
 \end{subfigure}
  \begin{subfigure}[h]{0.3\textwidth}
 \centering
 \includegraphics[width=\textwidth]{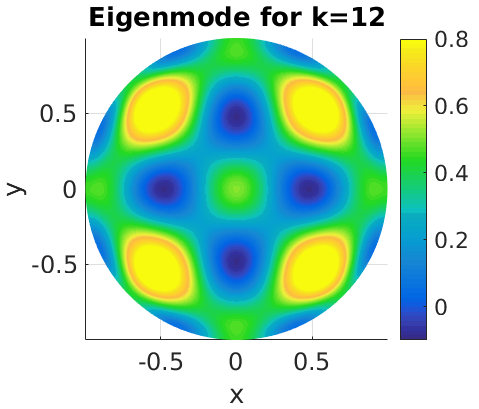}
  \caption{}
  \label{ktwelve}
 \end{subfigure}
  \begin{subfigure}[h]{0.3\textwidth}
 \centering
    \includegraphics[width=\textwidth]{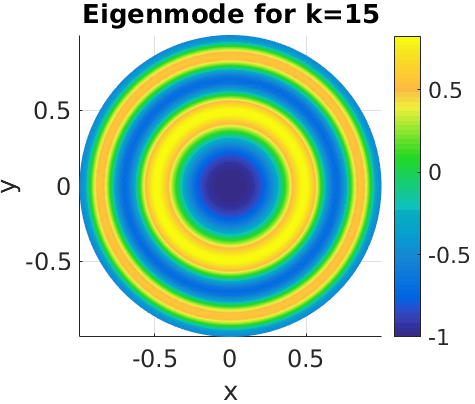}
  \caption{}
  \label{kfifteen}
 \end{subfigure}
 \caption{Colour encoded phaseplots of $w(r,\theta)$ for different values of $k$ indicated in each subtitle.}
  \label{eigenfunction}
\end{figure}
Note that in Figure \ref{nodal} different modes namely modes 2 and 3 result in the same eigenvalues that correspond to a different orientation of the nodal line depiction for $k=2$ and $k=3$. Similarly, one may find that there are a lot of such pairs of positive integers $k$ and $k+1$, that correspond to different orientations of the same nodal line depiction. 
This occurs, when the multiplicity of an eigenvalue $\eta_{n,k}$ is 2, and in fact there are numerous such pairs of $w_{n,k}(r,\theta)$ and $w_{n,k+1}(r,\theta)$, that correspond to the same $\eta_{n,k}$, but for different orientation of eigenmodes.
However applying the colour encoded representation using the (HSV) colour scheme indicates a distinction between the plots corresponding to each integer value for $k$ as shown in Figure \ref{eigenfunction}.

\begin{figure}
  \begin{center}
    \includegraphics[width=1\textwidth]{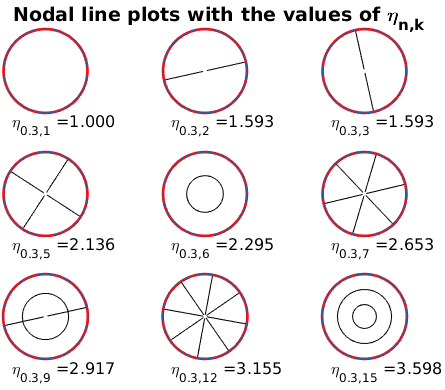}
  \end{center}
  \caption{Nodal lines depiction of the solution and the corresponding eigenvalues satisfying problem (\ref{eigen}).}
  \label{nodal}
\end{figure}
\subsection{Stability matrix and the characteristic polynomial}
The solution to system (\ref{polarsystem}) using separation of variables, can be written (with bars omitted from $\bar{u}$ and $\bar{v}$) as the product of the eigenfunctions of the diffusion operator $\triangle_p$ and $T(t)$ in the form 
\[\begin{split}
u(r,\theta,t) = \sum_{k=0}^\infty U_{n,k} \exp(\sigma_{n,k} t) R_{n,k}(r)\Theta_n(\theta),\\
v(r,\theta,t) = \sum_{k=0}^\infty V_{n,k} \exp(\sigma_{n,k} t) R_{n,k}(r)\Theta_n(\theta),
\end{split}\]
where $U_{n,k}$ and $V_{n,k}$ are the coefficients associated with the mode of the eigenfunctions in the infinite expansion. Substituting this form of solution in (\ref{vect}) and the steady state values in terms of the parameters $\alpha$ and $\beta$ for $(u_s,v_s)$, one obtains the fully linearised form of (\ref{polarsystem}) as a system of two algebraic equations of the form 
\begin{equation}
\sigma \left[\begin{array}{c}
     u \\
     v 
\end{array}\right]=-\eta_{n,k}^2\left[\begin{array}{cc}
     1&0  \\
     0&d 
\end{array}\right]\left[\begin{array}{c}
     u\\
     v 
\end{array}\right]+\gamma \left[\begin{array}{cc}
     \frac{\beta-\alpha}{\beta+\alpha}& (\beta+\alpha)^2  \\
     -\frac{2\beta}{\beta+\alpha}&-(\beta+\alpha)^2
\end{array}\right]\left[\begin{array}{c}
     u \\
     v
\end{array}\right],
\label{vect1}
\end{equation}
where $\sigma_{n,k} = \sigma$. In turn, this can be written as a two dimensional discrete eigenvalue problem of the form
\begin{equation}
\left[\begin{array}{cc}
     \gamma \frac{\beta-\alpha}{\beta+\alpha}-\eta^2_{n,k}& \gamma (\beta+\alpha)^2  \\
     -\gamma \frac{2\beta}{\beta+\alpha}&-\gamma(\beta+\alpha)^2-d\eta_{n,k}^2
\end{array}\right]\left[\begin{array}{c}
     u \\
     v
\end{array}\right]=\sigma \left[\begin{array}{c}
     u \\
     v 
\end{array}\right].
\label{vect2}
\end{equation}
The left-hand matrix in (\ref{vect2}) is referred to as the stability matrix \cite{book1, book7, paper12} for system (\ref{polarsystem}), with $\eta_{n,k}^2$ given by (\ref{eigenvalue}). In order to investigate the stability of the uniform steady state $(u_s,v_s)$, it is required to analyse the eigenvalues satisfying (\ref{vect2}), for which the characteristic polynomial takes the form of a quadratic equation in $\sigma$, written as
\begin{equation}
\left|\begin{array}{cc}
     \gamma \frac{\beta-\alpha}{\beta+\alpha}-\eta_{n,k}^2 -\sigma& \gamma (\beta+\alpha)^2  \\
     -\gamma \frac{2\beta}{\beta+\alpha}&-\gamma(\beta+\alpha)^2-d\eta_{n,k}^2-\sigma
\end{array}\right|=0.
\label{vect4}
\end{equation}
Let $\mathcal{T}(\alpha,\beta)$ and $\mathcal{D}(\alpha,\beta)$ respectively denote the trace and determinant of the stability matrix given by (\ref{vect2}), then the quadratic polynomial (\ref{vect4}) can be written in terms of $\mathcal{T}$ and $\mathcal{D}$ as
\begin{equation}
\sigma^2 - \mathcal{T}(\alpha, \beta) \sigma + \mathcal{D}(\alpha, \beta) = 0,
\label{charac}
\end{equation}
with $\mathcal{T}(\alpha, \beta)$ and $\mathcal{D}(\alpha, \beta)$ expressed by  
\begin{equation}\begin{cases}
\mathcal{T}(\alpha, \beta) =& \gamma \frac{\beta - \alpha  - (\beta +\alpha)^3}{\beta +\alpha}-(d+1)\eta_{n,k}^2,\\
\mathcal{D}(\alpha, \beta) =& \big(\gamma\frac{\beta-\alpha}{\beta+\alpha}-\eta_{n,k}^2\big)\big(-\gamma(\beta +\alpha)^2-(d+1)\eta_{n,k}^2\big)+2\gamma^2\beta(\beta+\alpha).
\end{cases}\label{detrace}\end{equation}
The roots of equation (\ref{charac}) in terms of $\mathcal{T}$ and $\mathcal{D}$ are $\sigma_{1,2}=\frac{\mathcal{T}\pm\sqrt{\mathcal{T}^2-4\mathcal{D}}}{2}$. Stability of the uniform steady state $(u_s,v_s)$ is determined by the signs of the two roots namely $\sigma_{1,2}$ if they are real. However, if $\sigma_{1,2}$ are a complex conjugate pair (with non-zero imaginary parts), then the sign of the real part is sufficient to predict the stability of the uniform steady state $(u_s,v_s)$. Since there are always exactly two roots on a complex plane for a quadratic equation, therefore, it is impossible with the same choice of parameters $\alpha$, $\beta$, $d$, $\gamma$ and $\eta^2_{n,k}$, for $\sigma_{1}$ to be real and $\sigma_2$ to be complex or vice versa. Therefore, a reasonable approach to encapsulate all the possibilities for the stability and types of the uniform steady state $(u_s,v_s)$ in light of parameters $\alpha$ and $\beta$ is to consider the cases when  $\sigma_{1,2}\in\mathbb{C}\backslash\mathbb{R}$ and $\sigma_{1,2}\in\mathbb{R}$. In each case the parameter space is rigorously analysed and the classification of the parameter plane $(\alpha,\beta)\in \mathbb{R}_+^2$ in relation to the diffusion parameter $d$ is studied. In light of such classification, the analysis is further extended to explore the effects of domain size on the existence of regions in parameter space that correspond to spatial and/or temporal bifurcations.
\section{Parameter spaces and bifurcation analysis}\label{bifur}
Bifurcation analysis of system (\ref{polarsystem}) is better conducted when the parameter plane $(\alpha,\beta)\in \mathbb{R}_+^2$ is appropriately partitioned for both of the cases when $\sigma_{1,2}\in\mathbb{C}\backslash\mathbb{R}$ as well as when $\sigma_{1,2}\in\mathbb{R}$. To obtain such a partition on the parameter plane, it is required to find the equations of the partitioning curves and these can be found through a detailed analysis of the expression for $\sigma_{1,2}$, which in turn requires to explore the domains of $\mathcal{T}$ and $\mathcal{D}$. 
\subsection{Equations of the partitioning curves}
Starting with the curve on the parameter plane $(\alpha,\beta)\in \mathbb{R}_+^2$ that forms a boundary for the region that corresponds to eigenvalues $\sigma_{1,2}$ containing non-zero imaginary part. It must be noted that the only possibility through which $\sigma_{1,2}$ can have a non-zero imaginary part is if the inequality $\mathcal{T}^2-4\mathcal{D}<0$ is true. It means that those parameter values $\alpha$ and $\beta$ satisfying the equation $\mathcal{T}^2(\alpha,\beta)=4\mathcal{D}(\alpha,\beta)$ must be lying on a partitioning curve that determines the boundary between the region on the parameter plane that corresponds to eigenvalues with non-zero imaginary part and that which corresponds to a pair of real eigenvalues. We therefore, state that the set of points on the parameter plane $(\alpha,\beta)\in \mathbb{R}_+^2$ satisfying the implicit equation
\begin{equation}\begin{split}
\Big(\gamma \frac{\beta - \alpha  - (\beta +\alpha)^3}{\beta +\alpha}-(d+1)\eta_{n,k}^2\Big)^2=&4\Big(\big(\gamma\frac{\beta-\alpha}{\beta+\alpha}-\eta_{n,k}^2\big)\big(-\gamma(\beta +\alpha)^2\\
&-(d+1)\eta_{n,k}^2\big)+2\gamma^2\beta(\beta+\alpha)\Big),
\label{part1}
\end{split}\end{equation}
forms the partitioning curve between the region that corresponds to a real pair of $\sigma_{1,2}$ and that corresponding to a complex conjugate pair of $\sigma_{1,2}$. For the solution of (\ref{part1}) refer to Section \ref{main}, where a numerical method is employed to find combinations of $\alpha,\beta\in\mathbb{R}_+$ on the plane $(\alpha,\beta)\in \mathbb{R}_+^2$ satisfying (\ref{part1}). It is worth noting that a combination of $\alpha$ and $\beta$ satisfying (\ref{part1}), entails that the expression for $\sigma_{1,2}$ possesses a repeated real root of the form $\sigma_1=\sigma_2=\frac{1}{2}\mathcal{T}$, whose sign will require to be explored in Section \ref{main} for real eigenvalues. Another important point to be made about the curve satisfying (\ref{part1}) is that it also partitions the parameter space for regions of spatial and temporal bifurcations. Because on one side of curve (\ref{part1}), both of the eigenvalues $\sigma_{1,2}$ are always real, which can never excite temporal instability of the uniform steady state $(u_s,v_s)$, whereas on the other side, the eigenvalues $\sigma_{1,2}$ are always a complex conjugate pair, which can never excite spatial instability. Therefore, on the side where $\sigma_{1,2}$ are a pair of real values, any instability that occurs in the dynamics of system (\ref{polarsystem}) will be strictly relevant to spatial variation, hence any pattern that system (\ref{polarsystem}) can evolve to will be strictly spatial with stable and invariant evolution in time. However, on the side where $\sigma_{1,2}$ have non-zero imaginary part, every possible instability in the dynamics of system (\ref{polarsystem}) will be strictly concerned with temporal periodicity, therefore any pattern that emerges from the dynamics of system (\ref{polarsystem}) is expected to be periodic along the time axis. This analysis raises the interesting question whether it is possible for the dynamics of (\ref{polarsystem}) to cause a spatially periodic pattern to undergo an unstable temporal periodicity as well? To answer this question, on the first sight it sounds absurd to claim on one hand that a specific choice of admissible $(\alpha,\beta)$ can either cause spatial or temporal instability, and not both at the same time, which is an intuitive claim to make. Because, any admissible choice of $(\alpha,\beta)$ either yields a pair of real $\sigma_{1,2}$ or a complex conjugate pair of $\sigma_{1,2}$, therefore, a fixed choice of $(\alpha,\beta)$ is not expected to yield a real $\sigma_1$ and a complex $\sigma_2$ or vice versa. Therefore, it is intuitive to presume that system (\ref{polarsystem}) should not admit the evolution of such dynamics, in which a combination of spatial and temporal instabilities can occur. However, the current study finds that this counter-intuitive behaviour is possible for system (\ref{polarsystem}) to exhibit, which is related to the existence and shift in the location of a partitioning curve on which the real part of $\sigma_{1,2}\in\mathbb{C}\backslash\mathbb{R}$ becomes zero during the course of the temporal evolution. Such a curve is referred to as the \textit{transcritical curve}, for which the implicit equation in terms of admissible $\alpha$ and $\beta$ is of the form
\begin{equation}
\gamma \frac{\beta - \alpha  - (\beta +\alpha)^3}{\beta +\alpha}=(d+1)\eta_{n,k}^2,
\label{part2}
\end{equation}
under the assumption that $\mathcal{T}^2(\alpha,\beta)-4\mathcal{D}(\alpha,\beta)<0$. It must be noted, that equations (\ref{part1}) and (\ref{part2}) are the only two equations that fully partition the parameter plane $(\alpha,\beta)\in \mathbb{R}_+^2$. However, one may expect the number of partitions subject to the types and stability of the uniform steady state $(u_s,v_s)$, on the parameter plane $(\alpha,\beta)\in \mathbb{R}_+^2$ to be four regions separated by three curves. These regions are obtained by sub-partitioning the region corresponding to a pair of real $\sigma_{1,2}$, into two sub-regions where $\sigma_{1,2}$ is a pair of real negative values and another where at least $\sigma_1$ or $\sigma_2$ is positive. The region on the parameter plane $(\alpha,\beta)\in \mathbb{R}_+^2$, that corresponds to complex eigenvalues can also be sub-divided into two regions, one where $\sigma_{1,2}$ is a complex conjugate pair with negative real part, and the other where $\sigma_{1,2}$ is a complex conjugate pair, but with a positive real part. The numerical solution of equation (\ref{part1}) will reveal that the region corresponding to complex eigenvalues is in fact bounded by two curves in the parameter plane $(\alpha,\beta)\in \mathbb{R}_+^2$, each of which satisfies (\ref{part1}). Therefore, (\ref{part1}) is implicitly the equation of two partitioning curves instead of one, which including (\ref{part2}) makes a total of three curves dividing the admissible parameter space into four different regions. It is worth bearing in mind, that subject to the types and stability of the uniform steady state, one obtains at most four regions separated by three curves, with the possibility that parameters such as $\gamma$, $d$ and $\eta_{n,k}^2$ may induce equations (\ref{part1}) and (\ref{part2}) in such a way that the number of partitioning curves may become less than three in total. Consequently this entails that a region corresponding to a certain type of bifurcation may completely disappear from the admissible parameter space. This kind of influence on the location and existence of the partitioning curves is quantitatively investigated, in particular, the effect of $\rho$ embedded in the expression for $\eta_{n,k}^2$ in (\ref{eigenvalue}) is explored to analyse its influence on the bifurcation of the uniform steady state.
\subsubsection{Analysis for the case of complex eigenvalues}
Before determining the region with complex eigenvalues on the admissible parameter plane $(\alpha,\beta)\in \mathbb{R}_+^2$ using a numerical treatment of (\ref{part1}), the real part of $\sigma_{1,2}$ is investigated analytically, when it is a complex conjugate pair. It can be noted that $\sigma_{1,2}$ can only become a pair of complex roots, if $(\alpha,\beta)$ satisfies the inequality
\begin{equation}
\mathcal{T}^2(\alpha,\beta)-4\mathcal{D}(\alpha,\beta)<0.
\label{comp}
\end{equation}
Given that (\ref{comp}) is satisfied, then the stability of the uniform steady steady state $(u_s,v_s)$ is decided purely by the sign of the real part of $\sigma_{1,2}$, which is the expression
\begin{equation}
\text{Re}(\sigma_{1,2})=\frac{1}{2}\Big(\gamma\frac{\beta-\alpha-(\beta +\alpha)^3}{\beta+\alpha}-(d+1)\eta_{n,k}^2\Big).
\label{real}
\end{equation}
If the sign of the expression given by (\ref{real}) is negative, simultaneously with assumption (\ref{comp}) satisfied, then it can be predicted that no choice of parameters can cause temporal instability in the dynamics of system (\ref{polarsystem}). Therefore, under the assumption (\ref{comp}), if the dynamics of system (\ref{polarsystem}) do exhibit diffusion-driven instability, it will be restricted to spatially periodic behaviour only, which uniformly converges to a temporal steady state, consequently one obtains spatial pattern that is invariant in time. The sign of the expression given in (\ref{real}) is further investigated to derive from it, relations between the parameter $\rho$ controlling the domain size and reaction-diffusion rates denoted by $\gamma$ and $d$ respectively. Given that assumption (\ref{comp}) is satisfied then the sign of expression (\ref{real}) is negative if parameters $\alpha$, $\beta$, $\gamma$ and $d$ satisfy the inequality 
\begin{equation}
\frac{\beta - \alpha  - (\beta +\alpha)^3}{\beta +\alpha}<\frac{(d+1)\eta_{n,k}^2}{\gamma},
\label{comp1}
\end{equation}
with $\eta_{n,k}^2$ defined by (\ref{eigenvalue}). Note that the expression on the left hand-side of (\ref{comp1}) is a bounded quantity by the constant value of 1 \cite{paper38}, for all the admissible choices of $(\alpha,\beta)\in\mathbb{R}_+$, therefore, substituting for $\eta_{1,2}^2$ in expression (\ref{eigenvalue}) and rearranging, it can be obtained that for inequality (\ref{comp1}) to remain true, it induces a restriction on the value of $\rho^2$, which is of the form 
\begin{equation}
\rho^2<\frac{4(d+1)(2k+1)(n+2k+1)(n+4k)}{\gamma(n+4k+2)}.
\label{comp2}
\end{equation}
Inequality (\ref{comp2}) conversely implies that so long as the radius $\rho$ of the disk shape domain $\Omega$ satisfies (\ref{comp2}), then the dynamics of system (\ref{polarsystem}) is guaranteed to exhibit global temporal stability, which also means that any possible instability in the dynamics must be restricted to spatial periodicity or spatial pattern. The formal proof of this claim is presented in Theorem \ref{theorem2}. This type of instability concerning space and not time is referred to as \textit{Turing instability} \cite{book1, book7}. On the other hand assuming that (\ref{comp}) is satisfied and using the upper bound of the quantity on the left hand-side of (\ref{comp1}) with $\eta_{n,k}^2$ as defined in (\ref{eigenvalue}), it can be shown that a necessary condition for the sign of expression (\ref{real}) to become positive is for $\rho$ to satisfy the inequality 
\begin{equation}
 \rho^2\geq\frac{4(d+1)(2k+1)(n+2k+1)(n+4k)}{\gamma(n+4k+2)}.
 \label{comp3}
\end{equation}
Conditions (\ref{comp2}) and (\ref{comp3}) both have quantitative influence on the location and topology of the partitioning curves obtained from the numerical solutions of (\ref{part1}) and (\ref{part2}) in the admissible parameter plane, namely  $(\alpha,\beta)\in \mathbb{R}_+^2$. System (\ref{polarsystem}) is restricted from any type of temporal bifurcation if the radius $\rho$ of the disk shape domain $\Omega$ is related to the reaction-diffusion parameters $\gamma$ and $d$, through inequality (\ref{comp2}). It also means that on a disk shape domain $\Omega$ with radius $\rho$ satisfying (\ref{comp2}), the dynamics of system (\ref{polarsystem}), either exhibit spatially periodic pattern or no pattern at all, both of which are globally stable in time. If the dynamics of a system become unstable along the time axis and exhibits temporal periodicity, then the system is said to undergo \textit{Hopf bifurcation} \cite{book12, book1, paper39, paper2}. If the real part of a pair of complex eigenvalues become zero, then the system is expected to exhibit oscillations with orbital periodicity. This behaviour is known as \textit{transcritical bifurcation} \cite{book1, paper2, paper39}. The consequences of the restriction on $\rho$ in the sense of bifurcation analysis means that, whenever the radius $\rho$ of a disk-shape domain is bounded by (\ref{comp2}) in terms of $\gamma$ and $d$, then the dynamics of system (\ref{polarsystem}) is guaranteed to forbid \textit{Hopf} and \textit{transcritical bifurcations}, only allowing for \textit{Turing instability} to occur. If system (\ref{polarsystem}) allows only Turing type instability to occur under condition (\ref{comp2}), it indicates that the eigenvalues $\sigma_{1,2}$ only become positive, when they are a pair of real values with zero imaginary parts. Therefore, it is also an indication that diffusion-driven instability is still possible, but it just becomes strictly spatial with (\ref{comp2}) satisfied. If the values of parameters $\gamma$ and $d$ are chosen such that inequality (\ref{comp3}) is satisfied, then the possibility of all three types of diffusion-driven instabilities exist on the admissible parameter plane  $(\alpha,\beta)\in \mathbb{R}_+^2$, namely Turing, Hopf and transcritical types of bifurcations. The influence from the area of $\Omega$ through the relationship (\ref{comp3}) of $\rho$ with the reaction-diffusion rates namely $\gamma$ and $d$ is summarised in Theorem \ref{theorem2} with a detailed sketch of the proof. 
\begin{theorem}[Hopf or transcritical bifurcation]
 Let $u$ and $v$ satisfy the non-dimensional reaction-diffusion system with {\it activator-depleted} reaction kinetics (\ref{polarsystem}) on a disk-shape domain $\Omega \subset \mathbb{R}^2$ with radius $\rho$ and positive real parameters $\gamma$, $d$, $\alpha$ and $\beta$.
 For the system to exhibit Hopf or transcritical bifurcation in the neighbourhood of the unique steady state $(u_s,v_s)=\big(\alpha+\beta, \frac{\beta}{(\alpha+\beta)^2}\big)$, the necessary condition on the radius $\rho$ of the disk-shape domain $\Omega \subset \mathbb{R}^2$ is that it must be sufficiently large satisfying
 \begin{equation}
 \rho\geq2\sqrt{\frac{(d+1)(2k+1)(n+2k+1)(n+4k)}{\gamma(n+4k+2)}},
 \label{comp4}
\end{equation}
 where $n\in\mathbb{R}\backslash\frac{1}{2}\mathbb{Z}$ is the associated order of the Bessel's equations and $k$ is any positive integer.
 \label{theorem2}
 \end{theorem}
\begin{proof}[Proof:]
For system (\ref{polarsystem}) to exhibit \textit{Hopf} or \textit{transcritical} bifurcations the eigenvalues of the stability matrix (\ref{vect2}) must have non-zero imaginary part with non-negative real part.
Consider the real part of $\sigma_{1,2}$, which is precisely given by $\frac{1}{2}\mathcal{T}(\alpha,\beta)$ under the assumption that the admissible choice of parameters $\alpha,\beta\in\mathbb{R}_+$ satisfies (\ref{comp}). When $\sigma_{1,2}\in\mathbb{C}\backslash\mathbb{R}$, then the stability of the uniform steady state $(u_s,v_s)=\big(\alpha+\beta, \frac{\beta}{(\alpha+\beta)^2}\big)$ is precisely determined by the sign of $\mathcal{T}(\alpha,\beta)$, which is given by 
\begin{equation}
\mathcal{T}(\alpha, \beta) = \gamma \frac{\beta - \alpha  - (\beta +\alpha)^3}{\beta +\alpha}-(d+1)\eta_{n,k}^2.
\label{real1}
\end{equation}
System (\ref{polarsystem}) undergoes Hopf or transcritical bifurcation if $\mathcal{T}(\alpha,\beta)\geq0$, given that the strict inequality (\ref{comp}) is satisfied, which can only hold true if $\mathcal{D}(\alpha,\beta)>0$. In (\ref{real1}) $\eta_{n,k}^2$ is given by
\begin{equation}
\eta_{n,k}^2=\frac{4(2k+1)(n+2k+1)(n+4k)}{\rho^2(n+4k+2)},
\label{eigenvalue1}
\end{equation}
where $n\in\mathbb{R}\backslash\frac{1}{2}\mathbb{Z}$ is the order of the associated Bessel's equation and $k\in\mathbb{N}$ is any positive integer. To show the condition on $\rho$ for Hopf or transcritical bifurcation, one may substitute (\ref{eigenvalue1}) into (\ref{real1}) and requiring the resulting quantity to be non-negative, which yields the inequality
\begin{equation}
\gamma \frac{\beta - \alpha  - (\beta +\alpha)^3}{\beta +\alpha}\geq\frac{4(d+1)(2k+1)(n+2k+1)(n+4k)}{\rho^2(n+4k+2)}.
\label{equatzero1}
\end{equation}
Noting that the left hand-side of (\ref{equatzero1}) can be written as the difference between two non-negative functions $f_1(\alpha,\beta)$ and $f_2(\alpha,\beta)$ in the form $\gamma\big(f_1(\alpha,\beta)-f_2(\alpha,\beta)\big)$, where $f_1$ and $f_2$ are given by
\begin{equation}
f_1(\alpha,\beta)=\frac{\beta}{\alpha+\beta}, \qquad f_2(\alpha,\beta)=\frac{\alpha+(\alpha+\beta)^3}{\alpha+\beta}.
\label{f1f2}
\end{equation}
Note also that $\rho^2$ resides in the denominator of the right hand-side of (\ref{equatzero1}) and parameter $\gamma$ is multiplied by the expression on the left hand-side. In order to find what this inequality induces on the relationship between parameters $\gamma$, $d$ and $\rho$, it is essential to analyse the supremum and infemum of $f_1(\alpha,\beta)$ and $f_2(\alpha,\beta)$ within their respective domains which is $(\alpha,\beta)\in[0,\infty)\times[0,\infty)$. The range for $f_1(\alpha,\beta)$ and $f_2(\alpha,\beta)$ are independently analysed to find the supremum of the expression on the left of (\ref{equatzero1}). Starting with $f_1(\alpha,\beta)$, which is bounded below and above in the domain $(\alpha,\beta)\in[0,\infty)\times [0,\infty)$, we have
$\sup_{\alpha,\beta \in \mathbb{R}_+} f_1(\alpha,\beta) = 1$, and 
the $\inf_{\alpha,\beta \in \mathbb{R}_+} f_1(\alpha,\beta) = 0$ for all $ \alpha,\beta \in \mathbb{R}_+.$
Similarly considering the expression for $f_2(\alpha,\beta)$, we have
$\sup_{\alpha,\beta \in \mathbb{R}_+} f_2(\alpha,\beta) = \infty, $ and 
the $\inf_{\alpha,\beta \in \mathbb{R}_+} f_2(\alpha,\beta) = 0$, for all $\alpha, \beta \in \mathbb{R}_+.$
Since the ranges of both $f_1(\alpha,\beta)$ and $f_2(\alpha,\beta)$ are non-negative within their respective domains, therefore the supremum of their difference is determined by the supremum of the function with positive sign, which is $\sup_{\alpha,\beta\in\mathbb{R}_+}f_1(\alpha,\beta)=1$. Therefore, inequality (\ref{equatzero1}) takes the form
\[\begin{split}
\frac{4(d+1)(2k+1)(n+2k+1)(n+4k)}{\rho^2(n+4k+2)}&\leq\gamma\frac{\beta - \alpha  - (\beta +\alpha)^3}{\beta +\alpha} \\
& \leq \gamma \sup_{\alpha,\beta\in \mathbb{R}_+} \big(f_1(\alpha,\beta)-f_2(\alpha,\beta)\big) \\
&=\gamma \sup_{\alpha,\beta \in \mathbb{R}_+} f_1(\alpha,\beta)  = \gamma,
\end{split}\]
which by rearranging and writing the inequality for $\rho$ in terms of everything else, yields the desired statement of Theorem \ref{theorem2}, which is condition (\ref{comp4}).
\end{proof}
The claim of Theorem \ref{theorem2}, is also numerically verified by showing that a region in the admissible parameter plane that corresponds to Hopf or transcritical bifurcations emerges only if radius $\rho$ of a disk shape domain $\Omega$ is sufficiently large satisfying inequality (\ref{comp4}). Otherwise, no choice of parameters exist in the admissible parameter plane $(\alpha,\beta)\in\mathbb{R}_+^2$, allowing the dynamics of (\ref{polarsystem}) to exhibit Hopf or transcritical bifurcation.

\subsubsection{Analysis for the case of real eigenvalues}
The eigenvalues $\sigma_{1,2}$ are both real if the discriminant of the roots is either zero or positive, which in turn means that the relationship between $\mathcal{T}(\alpha,\beta)$ and $\mathcal{D}(\alpha,\beta)$ is such that 
\begin{equation}
\mathcal{T}^2(\alpha,\beta) \geq 4\mathcal{D}(\alpha,\beta).
\label{condreal}
\end{equation}
The equal case of (\ref{condreal}) is looked at first, where we have
\begin{equation}
\mathcal{T}^2(\alpha,\beta) = 4\mathcal{D}(\alpha,\beta),
\label{condequal}
\end{equation}
which implies that the discriminant is zero, hence the roots are repeated real values of the form $\sigma_1 = \sigma_2\in\mathbb{R}$, given by
\begin{equation}
\sigma_1=\sigma_2 = \frac{1}{2}\Big(\gamma\frac{\beta - \alpha  - (\beta +\alpha)^3}{\beta +\alpha}-\frac{4(d+1)(2k+1)(n+2k+1)(n+4k)}{\rho^2(n+4k+2)}\Big).
\label{repeated}
\end{equation}
When $\alpha$ and $\beta$ satisfy condition (\ref{condequal}), the stability of the steady state is determined by the sign of the root itself. 
The expression given by (\ref{repeated}) can be easily shown to be negative if the radius $\rho$ of the disk-shape domain $\Omega$ satisfies the inequality
\begin{equation}
\rho<2\sqrt{\frac{(\alpha+\beta)(d+1)(2k+1)(n+2k+1)(n+4k)}{\gamma(\beta-\alpha-(\alpha+\beta)^3)(n+4k+2)}}.
\label{repeatneg}
\end{equation}
 Otherwise, the repeated root is positive provided that $\rho$ satisfies
\begin{equation}
\rho>2\sqrt{\frac{(\alpha+\beta)(d+1)(2k+1)(n+2k+1)(n+4k)}{\gamma(\beta-\alpha-(\alpha+\beta)^3)(n+4k+2)}}.
\label{repeatpos}
\end{equation}
Analysing (\ref{repeatneg}) and (\ref{repeatpos}) carefully, it can be observed that the only terms that can possibly invalidate the inequalities are in the denominator of the right hand-side, namely the expression $\beta-\alpha-(\beta+\alpha)^3$. Therefore, a restriction is required to be stated on this term to ensure that the radius $\rho$ of $\Omega$ is not compared against an imaginary number, such a restriction is 
\begin{equation}
\beta > \alpha+(\beta+\alpha)^3.
\label{rest}
\end{equation}
It must be noted that (\ref{rest}) is the same restriction on the parameter choice obtained for the case of repeated real eigenvalues in the absence of diffusion \cite{paper38}. By further comparing with (\ref{repeatpos}), it can be noted that it is very similar to condition (\ref{comp4}) of Theorem \ref{theorem2}, except that (\ref{comp4}) is free from any dependence of the parameters $\alpha$ and $\beta$. This makes (\ref{comp4}) a sharper version of (\ref{repeatpos}) in the sense, that the curve satisfying (\ref{condequal}) subject to condition (\ref{repeatpos}) must be the one forming the boundary of the region on the admissible parameter plane, that corresponds to complex eigenvalues $\sigma_{1,2}$ with positive real parts, which is the region for Hopf bifurcation. Therefore, the region of the admissible parameter plane that corresponds to Hopf bifurcation is on one side bounded by curve (\ref{condequal}) and on the other side it is bounded by the curve satisfying (\ref{part2}) under their respective assumptions. In Section \ref{main} it is verified to be the case by the numerical computation of the partitioning curve (\ref{part1}), which is the same curve (\ref{condequal}). This analysis motivates to explore the possibility of similar comparison between the conditions (\ref{repeatneg}) and (\ref{comp2}). A reasonable intuition behind this comparison is that the sub-region on the admissible parameter plane that corresponds to complex eigenvalues with negative real parts must be bounded by curve (\ref{condequal}) subject to condition (\ref{repeatneg}), outside of which every possible choice of parameters $\alpha$ and $\beta$ will guarantee the eigenvalues $\sigma_{1,2}$ to be a pair of distinct real values, which promotes the necessity to state and prove Theorem \ref{theorem3}.
\begin{theorem}[Turing type diffusion-driven instability]
Let $u$ and $v$ satisfy the non-dimensional reaction-diffusion system with {\it activator-depleted} reaction kinetics (\ref{polarsystem}) on a disk-shape domain $\Omega \subset \mathbb{R}^2$ with radius $\rho$ and positive real parameters $\gamma$, $d$, $\alpha$ and $\beta$.
 Given that the radius $\rho$ of domain $\Omega \subset \mathbb{R}^2$ satisfies the inequality 
 \begin{equation}
     \rho<2\sqrt{\frac{(d+1)(2k+1)(n+2k+1)(n+4k)}{\gamma(n+4k+2)}},
     \label{condtur}
 \end{equation}
 where $n\in\mathbb{R}\backslash\frac{1}{2}\mathbb{Z}$ is the associated order of the Bessel's equations and $k$ is any positive integer, then for all $\alpha, \beta \in \mathbb{R}_+$ in the neighbourhood of the unique steady state $(u_s,v_s)=\big(\alpha+\beta, \frac{\beta}{(\alpha+\beta)^2}\big)$ the diffusion driven instability is restricted to Turing type only, forbidding the existence of Hopf and transcritical bifurcations.
 \label{theorem3}
 \end{theorem}
\begin{proof}[Proof:]
 The strategy of this proof is through detailed analysis of the real part of the eigenvalues of the linearised system, when the eigenvalues are a complex conjugate pair. This can be done through studying the surface $\mathcal{T}(\alpha,\beta)$, and finding that it has a unique extremum point at $(0,0)$.  The method of the second derivative test and Hessian matrix is used to determine the type of this extremum. Upon finding its type, then the monotonicity of $\mathcal{T}(\alpha,\beta)$ is analysed in the neighbourhood of the extremum point in both directions $\alpha$ and $\beta$. The monotonicity analysis and the type of the extremum leads to proving the claim of the theorem.
 
The eigenvalues $\sigma_{1,2}$, in terms of trace $\mathcal{T}(\alpha,\beta)$ and determinant $\mathcal{D}(\alpha,\beta)$ are given by
 $\sigma_{1,2} = \frac{1}{2}\mathcal{T}(\alpha,\beta) \pm \frac{1}{2}\sqrt{\mathcal{T}^2(\alpha,\beta) -4\mathcal{D}(\alpha,\beta)}$,
 where 
 \[\begin{split}
\mathcal{T}(\alpha,\beta)=& \gamma \frac{\beta - \alpha  - (\beta +\alpha)^3}{\beta +\alpha}-(d+1)\eta_{n,k}^2,\\
 \mathcal{D}(\alpha,\beta) =&\Big(\gamma \frac{\beta-\alpha}{\beta+\alpha}-\eta_{n,k}^2\Big)\Big(-\gamma(\beta +\alpha)^2-d\eta_{n,k}^2\Big)+2\gamma^2\beta(\beta+\alpha),
 \end{split}\]
 with $\eta_{n,k}^2$ as given by (\ref{eigenvalue}). It can be immediately observed that in the neighbourhood of $(u_s,v_s)$ for the system to exhibit Hopf or transcritical bifurcation the discriminant of the characteristic polynomial must satisfy the inequality
 $\mathcal{T}^2(\alpha,\beta) -4\mathcal{D}(\alpha,\beta)<0$.
 Therefore, the stability and type of the steady state $(u_s,v_s)$ in this case is determined by the sign of the real part of $\sigma_{1,2}$. The aim is to investigate $\mathcal{T}(\alpha,\beta)$ and derive from it condition (\ref{condtur}) on $\rho$ as a requirement for $\mathcal{T}(\alpha,\beta)$ to be negative for all strictly positive choices of $\gamma$, $\alpha$, $\beta$ and $d$.
 The first derivative test is used on $\mathcal{T}(\alpha,\beta)$ to find the stationary points of $\mathcal{T}(\alpha,\beta)$ on the domain $[0,\infty) \times [0,\infty)$.
 All stationary points of $\mathcal{T}(\alpha,\beta)$ must satisfy
 $\frac{\partial \mathcal{T}}{\partial \alpha} = - \gamma\frac{2(\alpha+\beta)^3+2\beta}{(\alpha+\beta)^2} = 0$,
 which is true if and only if 
 \begin{equation}
  (\alpha +\beta)^3+\beta = 0.
  \label{first}
 \end{equation}
 Similarly all stationary points of $\mathcal{T}(\alpha,\beta)$ must also satisfy
 $\frac{\partial \mathcal{T}}{\partial \beta} = - \gamma \frac{2(\alpha+\beta)^3-2\alpha}{(\alpha+\beta)^2} = 0$,
 which implies 
  \begin{equation}
  (\alpha +\beta)^3-\alpha = 0.
  \label{second}
 \end{equation}
 The system of nonlinear algebraic equations obtained from (\ref{first}) and (\ref{second}) has a unique solution namely $\alpha=0$ and $\beta = 0$.
 Therefore, $\mathcal{T}(\alpha,\beta)$ has a unique stationary point at the origin. The type of this stationary point is determined by the second derivative test for which the Hessian matrix $H(\mathcal{T}(\alpha,\beta))$ must be computed and evaluated at the point $(0,0)$. A similar approach to that used in \cite{paper38} is applied to analyse the type of the unique stationary point.
 
It is clear that the entries of $H$ upon direct evaluation at the point $(0,0)$ are undefined. This is treated by using L'Hopital's rule.  L'Hopital's rule sometimes does not work for functions of two variables defined on cartesian coordinates, therefore a transformation of the entries to polar coordinates is applied. We will exploit this technique to express the Hessian matrix in polar coordinates and differentiate accordingly. 
The entries of $H$ are transformed to polar coordinates using $\alpha = \hat{r}\cos(\hat{\theta})$ and $\beta = \hat{r}\sin(\hat{\theta})$, so the rule can be applied by taking the $\lim_{r\rightarrow 0} H$.
Using $(\hat{r},\hat{\theta})$ coordinates the entries of $H$ take the following form
\begin{equation}
H(\mathcal{T}(\hat{r},\hat{\theta}))|_{\hat{r} = 0}=-\gamma \left[\begin{array}{cc}
   \frac{4\hat{r}\sin\hat{\theta}-2r^3(\cos\hat{\theta}+\sin\hat{\theta})^3}{r^3(\cos\hat{\theta}+\sin\hat{\theta})^3}  &  \frac{2\hat{r}^3(\cos\hat{\theta}+\sin\hat{\theta})^3+2\hat{r}(\cos\hat{\theta}-\sin\hat{\theta})}{\hat{r}^3(\cos\hat{\theta}+\sin\hat{\theta})^3}\\
      \frac{2\hat{r}^3(\cos\hat{\theta}+\sin\hat{\theta})^3+2\hat{r}(\cos\hat{\theta}-\sin\hat{\theta})}{\hat{r}^3(\cos\hat{\theta}+\sin\hat{\theta})^3}&\frac{4\hat{r}\cos\hat{\theta}+2\hat{r}^3(\cos\hat{\theta}+\sin\hat{\theta})^3}{\hat{r}^3(\cos\hat{\theta}+\sin\hat{\theta})^3} 
\end{array}\right]_{\hat{r}=0}.
\label{hess1}
\end{equation}
L'Hopital's rule is applied to each entry of $H$ separately and the $\lim_{\hat{r} \rightarrow 0} H_{ij}(\mathcal{T}(\hat{r},\hat{\theta}))$ is computed for $i,j=1,2$. Starting with the entry $H_{11}$ and cancelling $\hat{r}$, the expression takes the form
 \[
 \begin{split}
 \lim_{\hat{r} \rightarrow 0} H_{11} =\lim_{\hat{r} \rightarrow 0}\frac{4\sin\hat{\theta}-2\hat{r}^2(\cos\hat{\theta}+\sin\hat{\theta})^3}{\hat{r}^2(\cos\hat{\theta}+\sin\hat{\theta})^3}.
 \end{split}
 \]
 Let $\mathcal{T}_1(\hat{r},\hat{\theta})$ and $\mathcal{T}_2(\hat{r},\hat{\theta})$ respectively denote the numerator and the denominator of the expression for $H_{11}$, then the application of L'Hopital's rule suggests that 
 \[\begin{split}
 \lim_{r \rightarrow 0} H_{11}(\mathcal{T}(\hat{r},\hat{\theta})) &= \lim_{\hat{r} \rightarrow 0} \frac{\mathcal{T}_1(\hat{r},\theta)}{\mathcal{T}_2(\hat{r},\theta)}= \frac{\lim_{\hat{r} \rightarrow 0} \frac{d\mathcal{T}_1}{d\hat{r}}(\hat{r},\hat{\theta})}{\lim_{\hat{r} \rightarrow 0} \frac{d\mathcal{T}_2}{d\hat{r}}(\hat{r},\hat{\theta})}= \lim_{\hat{r} \rightarrow 0}\frac{ -4\hat{r}(\cos \hat{\theta} + \sin \hat{\theta})^3}{2 \hat{r} (\cos \hat{\theta} +\sin \hat{\theta})^3}= -2.
 \end{split}\]
 Applying the same procedure for $H_{12}$, $H_{21}$ and $H_{22}$, all the entries of $H$ are computed and given by
 \begin{equation}
H(\mathcal{T}(\alpha,\beta))|_{(0,0)}=-\gamma\left[\begin{array}{cc}
   -2  &  2 \\
  2  &  2
\end{array}\right].
\label{hessnum}
\end{equation}
Since the $\text{det}(H) = -8 \gamma^2< 0$, therefore, the second derivative test suggests that $(0,0)$ is a saddle point of $\mathcal{T}(\alpha,\beta)$. Since it was previously shown that $\mathcal{T}(\alpha,\beta)$ attains a unique stationary point in the domain $[0,\infty)\times[0,\infty)$, i.e. by solving the equations (\ref{first}) and (\ref{second}), therefore, if $(0,0)$ was a maximum and $\mathcal{T}(0,0)<0$, this would suggest that, whenever $\sigma_{1,2}$ has a non-zero imaginary part then $Re(\sigma_{1,2})<0$ regardless of the choice of $d$, $\gamma$ and $\rho$, however due to fact that $(0,0)$ is a saddle point, it remains to show that $\mathcal{T}(\alpha,\beta)$ is negative at $(0,0)$ and its first derivatives in the neighbourhood of $(0,0)$ of $\mathcal{T}(\alpha,0)$ and $\mathcal{T}(0,\beta)$ for positive values of $\alpha$ and $\beta$ in both directions are negative and do not change sign. Let $\mathcal{T}_0(\alpha)$ and $\mathcal{T}_0(\beta)$ denote the curves for constants $\beta=0$ and $\alpha=0$ respectively on the surface defined by $\mathcal{T}(\alpha,\beta)$, then 
\[\begin{split}
\mathcal{T}_0(\alpha) = \lim_{\beta \rightarrow 0}\mathcal{T}(\alpha,\beta) = -\gamma(1+\alpha^2)-(d+1)\eta_{n,k}^2,\\
\mathcal{T}_0(\beta) = \lim_{\alpha \rightarrow 0}\mathcal{T}(\alpha,\beta) = \gamma(1-\beta^2)-(d+1)\eta_{n,k}^2.
\end{split}\]
The expression for $\mathcal{T}_0(\alpha)$ clearly satisfy that it is negative at $\alpha=0$ and its first derivative in the direction of $\alpha$ is
$ \frac{d\mathcal{T}_0(\alpha)}{d\alpha} = -2\gamma\alpha <0$ for all  $\gamma, \alpha \in [0,\infty)$.
The expression for $\mathcal{T}_0(\beta)$ however is not trivially negative for all values, since the sign of the constant $\gamma$ in the expression is positive, which if computed at $\beta=0$, with substituting (\ref{eigenvalue}) for $\eta_{n,k}^2$ leads to the desired condition (\ref{condtur})
\begin{equation}
\mathcal{T}_0(\beta)\big|_{\beta=0} = \gamma-(d+1)\eta_{n,k}^2<0 \implies \rho<2\sqrt{\frac{(d+1)(2k+1)(n+2k+1)(n+4k)}{\gamma(n+4k+2)}}. \notag 
\end{equation}
It has been shown that the condition (\ref{condtur}) is necessary for $\mathcal{T}(\alpha,\beta)$ to be negative at the unique stationary point namely $(0,0)$, it remains to show that the first derivative of $\frac{d\mathcal{T}_0}{d\beta}(\beta)<0$,
$\frac{d\mathcal{T}_0}{d\beta} = -2\gamma \beta <0$  for all $ \gamma,\beta \in [0,\infty)$
which completes the proof.
 \end{proof}
 The region where the eigenvalues are repeated real roots is defined  by the implicit curves in the parameter space satisfying (\ref{condequal}), these curves are computed numerically in Section \ref{main}. These curves form the boundary between the regions of complex and real eigenvalues. Varying the diffusion rate $d$ causes a shift in the location of the curves indicating clearly regions that are subject to diffusion-driven instability.
The remaining case to look at is when both eigenvalues are real and distinct. This happens if $\alpha$ and $\beta$ are chosen such that the strict inequality case of (\ref{condreal}) is satisfied. This case corresponds to the diffusion-driven instability \textit{Turing type} only, because both eigenvalues are real and distinct. 
\subsection{Interpretation of the dynamics for the case of real eigenvalues}
If both eigenvalues are negative real values and distinct, then the system is spatially as well as temporally stable, the dynamics will achieve no patterns, hence the system returns to the uniform constant steady state $(u_s,v_s)$ as time grows, [see Section \ref{fem} Figure \ref{stable} (a)] with no effect from diffusion. If the eigenvalues are both real with different signs, then the type of instability caused by diffusion is spatially periodic or oscillatory in space, because this case corresponds to the steady state becoming a saddle point. If both eigenvalues are positive real values and distinct, then the dynamics are expected to exhibit a spatially periodic pattern, in the form of stripes or spots. 
\section{Solution of partitioning curves and numerical verification}\label{main}
This section mainly focuses on simulating the numerical solutions of (\ref{part1}) and (\ref{part2}), furthermore, a pictorial representation of the implicit curves satisfying (\ref{part1}) and (\ref{part2}) on the admissible parameter plane $(\alpha,\beta)\in\mathbb{R}_+^2$ is obtained in light of which the full classification of the admissible parameter plane is presented. Numerical solutions of the implicit partitioning curves satisfying (\ref{part1}) and (\ref{part2}) are computed subject to both conditions respectively given by (\ref{comp4}) and (\ref{condtur}) on the radius $\rho$ of $\Omega$ in terms of parameters $d$ and $\gamma$. 
The relationship between radius $\rho$ of the domain and parameters $d$ and $\gamma$ is shown to be in agreement with theoretical predictions presented.
 \subsection{Numerical solution of the partitioning curves}\label{numsol}
 Using algebraic manipulation, expanding the brackets and rearranging one can easily show that equation (\ref{part1}) can be written as a six degree implicit polynomial in the variable $\beta$, where the coefficients of such polynomial depend on all the remaining parameters namely $\alpha$, $d$, $\gamma$ and the eigenvalues $\eta^2_{n,k}$ of the diffusion operator $\triangle_p$. Let $\psi(\alpha,\beta)$ denote the six degree polynomial obtained from manipulating and rearranging (\ref{part1}), then finding the implicit solution satisfying equation (\ref{part1}) is equivalent to finding all the roots of the six degree polynimial equation $\psi(\alpha,\beta)=0$, which according to the fundamental theorem of algebra there exist at most six distinct roots \cite{paper38, paper42}. 
 
 Solutions of the implicit curves satisfying (\ref{part1}) are obtained by constructing a two-dimensional quadrilateral mesh on a rectangular domain $P=[0,\alpha_{max}]\times[0,\beta_{max}]$, where $\alpha_{max}$ and $\beta_{max}$ are the maximum positive real values in the respective directions of $\alpha$ and $\beta$ beyond which in the admissible parameter space, the uniform steady state $(u_s,v_s)$ neither changes type nor does it change stability. Domain $P$ is discretised by $N$ points in both directions of $\alpha$ and $\beta$, where $N$ is a positive integer, which leads to a rectangular mesh of $(N-1)\times(N-1)$ cells, each of size $\frac{\alpha_{max}}{N} \times \frac{\beta_{max}}{N}$, with $N^2$ points in $P$. 
To obtain the implicit solutions for (\ref{part1}), at every mesh point $\alpha_i$ in the direction the parameter-coordinate $\alpha $, the roots of the six degree polynomial in $\beta$ namely $\psi(\alpha_i,\beta)=0$, denoted by $\psi_i(\beta)=0$ are computed using the Matlab command `$roots$'. For every fixed mesh point $\alpha_i$ and fixed parameters $d$, $\gamma$ and eigenvalue $\eta_{n,k}$ one obtains $\psi_i(\beta)$, in the form
 \begin{equation}
 \psi_i(\beta)=C_0(\alpha_i)+C_1(\alpha_i)\beta+C_2(\alpha_i)\beta^2+C_3(\alpha_i)\beta^3+C_4(\alpha_i)\beta^4+C_5(\alpha_i)\beta^5+C_6(\alpha_i)\beta^6,
 \label{polyn1}
 \end{equation}
 where the expressions for the coefficients are given by
 \begin{equation*}
\begin{split}
C_0(\alpha_i) =& \alpha_i^2\gamma-2\alpha_i^2\gamma\eta_{n,k}^2+2\alpha_i^4\gamma+\alpha_i^2\eta_{n,k}^4-2\alpha_i^2d\gamma\eta_{n,k}^2-4\alpha_i^4\gamma^2+2\alpha_i^2d\eta_{n,k}^4-4\alpha_i^3\eta_{n,k}^4\\
&+2\alpha_i^4\gamma\eta_{n,k}^2+\alpha_i^6\gamma+\alpha_i^2d^2\eta_{n,k}^4+2\alpha_i^4d\eta_{n,k}^2-4\alpha_i^5\gamma\eta_{n,k}^2,
\end{split}
\end{equation*}
\begin{equation*}
\begin{split}
C_1(\alpha_i) =& 8\alpha_i^3\gamma d\eta_{n,k}^2-20\alpha_i^4\gamma\eta_{n,k}^2-12\alpha_i^2d\eta_{n,k}^4+2\alpha_i^2d\eta_{n,k}^4+6\alpha_i^5\gamma+8\alpha_i^3\gamma\eta_{n,k}^2-12\alpha_i^2\eta_{n,k}^4\\
&+4\alpha_id\eta_{n,k}^4+2\alpha_i\eta_{n,k}^4+4\alpha_i^3\gamma -2\alpha\gamma, 
\end{split}
\end{equation*}
\begin{equation*}
\begin{split}
C_2(\alpha_i)=&\gamma+2\gamma\eta_{n,k}^2+\gamma\eta_{n,k}^2+\eta_{n,k}^4+2d\gamma\eta_{n,k}^2+24\alpha_i^2\gamma^2+2d\eta_{n,k}^4-12\alpha_i\eta_{n,k}^4+12\alpha_i^2\gamma\eta_{n,k}^2\\
&+15\alpha_i^4\gamma+d^2\eta_{n,k}^4-12\alpha d\eta_{n,k}^4+12\alpha_i^2d\gamma\eta_{n,k}^2-40\alpha_i^3\gamma\eta_{n,k}^2,
\end{split}
\end{equation*}
\begin{equation*}
\begin{split}
C_3(\alpha_i) =&32\alpha_i\gamma^2-4\alpha_i\gamma-4\eta_{n,k}^4+8\alpha_i\gamma\eta_{n,k}^2+20\alpha_i^3\gamma-4d\eta_{n,k}^4+8\alpha_id\gamma\eta_{n,k}^2-40\alpha_i^2\gamma\eta_{n,k}^2,
\end{split}
\end{equation*}
\begin{equation*}
\begin{split}
C_4(\alpha_i)=&12\alpha_i^2-2\gamma+2\gamma\eta_{n,k}^2+15\gamma\alpha_i^2+2d\gamma\eta_{n,k}^2-20\alpha_i\gamma\eta_{n,k}^2,
\end{split}
\end{equation*}
\begin{equation*}
\begin{split}
C_5(\alpha_i)=6\alpha_i\gamma-4\gamma\eta_{n,k}^2, \quad \text{and} \quad C_6(\alpha_i)=&\gamma.
\end{split}
\end{equation*}
 In the expressions for the coefficients of (\ref{polyn1}), the values of $\eta_{n,k}$ are used as given by (\ref{eigenvalue}). Each of the $N$ polynomial equations in the direction of $\alpha$ has at most six roots for every fixed $\alpha_i$. Let $\beta_j$ for $j\in\{1,2,3,4,5,6\}$ denote all the six roots of $\psi_i(\beta)=0$, then the point $(\alpha_i,\beta_j)$ obtained by pairing each index $i=1,2,3,...,N$ with each root $\beta_j$ lies on the curve satisfying (\ref{part1}). A similar procedure is applied to obtain the solutions of the second partitioning curve satisfying (\ref{part2}), that corresponds to a curve in the admissible parameter space on which the real part of $\sigma_{1,2}\in\mathbb{C}\backslash\mathbb{R}$ is zero and it corresponds to the uniform steady state $(u_s,v_s)$ undergoing a transcritical bifurcation. Using a similar notation as for $\psi_i(\beta)=0$, the solutions to (\ref{part2}) are equivalent to finding all the roots of an implicit cubic polynomial in $\beta$ of the form $\phi_i(\beta)=0$, where $\phi_i(\beta)$ is given by 
 \begin{equation}
 \phi_i(\beta)=C_0(\alpha_i)+C_1(\alpha_i)\beta+C_2(\alpha_i)\beta^2+C_3(\alpha_i)\beta^3,
 \label{polyn2}
 \end{equation}
 with expressions for the coefficients of (\ref{polyn2}), taking the forms $C_0(\alpha_i)=-\alpha_i^3\gamma-\alpha_id\eta_{n,k}^2-\alpha_i\eta_{n,k}^2-\alpha_i\gamma$,
     $C_1(\alpha_i)=\gamma-\eta_{n,k}^2-d\eta_{n,k}^2-3\alpha_i^2\gamma$,
     $C_2(\alpha_i)=-3\alpha_i\gamma,$ 
     $C_3(\alpha_i)=-\gamma$ and $\eta_{n,k}$ as given by (\ref{eigenvalue}). Parameters $\alpha$ and $\beta$ resemble strictly positive real values in system (\ref{polarsystem}), therefore, upon computing all the roots of polynomials (\ref{polyn1}) and (\ref{polyn2}), the algorithm is instructed to record only roots satisfying $\beta_j \in \mathbb{R}_+$, and any roots that are either negative real or have non-zero imaginary part, are discarded. It is due to this constraint on the algorithm that gives all the admissible solutions satisfying (\ref{part1}) on the parameter space. In Figure \ref{realcomplex} the number of intersections of a vertical straight line in the direction of $\beta$ for a fixed value of $\alpha$, indicates the number of positive real roots of polynomial (\ref{polyn1}). 
 This algorithm is executed for five different values of $d$ to obtain the solutions of (\ref{part1}) and (\ref{part2}) under conditions (\ref{comp4}) and (\ref{condtur}) on the radius $\rho$. The shift and existence of the partitioning curves satisfying (\ref{part1}) and (\ref{part2}) are analysed subject to the variation of parameter $d$. It is worth noting that under the current analysis the variation of $d$ suffices to disclose every desirable insight one wishes to obtain about the variation of the parameter $\gamma$. This is due to their reciprocal locations in the expression on the right hand-sides of conditions (\ref{comp4}) and (\ref{condtur}). In other words, increasing parameter $d$ is equivalent to decreasing the parameter $\gamma$ and vice versa. Therefore, for fixed value of $\gamma$ it suffices to study the variation of $d$. Using condition (\ref{comp4}) of Theorem \ref{theorem2}, the variation of the diffusion coefficient is analysed for five different values of $d$ and Figure \ref{realcomplex} (a) shows the shift of the solutions of (\ref{part1}). It was shown that solution of (\ref{part1}) determines a partition of the region in the parameter space that corresponds to $\sigma_{1,2}$ to be a pair of complex values and it can be seen that under condition (\ref{comp4}), as the value of the parameter $d$ increases the region that corresponds to complex eigenvalues gradually decreases, however so long as condition (\ref{comp4}) remains satisfied, the vertical axis $\beta$ continues to have at least two distinct intercepts of the curves satisfying (\ref{part1}) for the same value of $d$. It is this behaviour of the real-complex eigenvalues partitioning curve that either preserves or vanishes the existence of a sub-partition for positive and negative real parts of a complex pair of roots. The five values of parameter $d$ are clearly indicated on the curves in Figure \ref{realcomplex}. It would be reasonable to use exactly the same range for the variational values of $d$ under both conditions (\ref{comp4}) and (\ref{condtur}), however, it was noted that, when the domain size is restricted by (\ref{condtur}) then the same values used for varying $d$ in Figure \ref{realcomplex} (a) invalidate inequality (\ref{condtur}), therefore, the range of variational values for parameter $d$ under condition (\ref{condtur}) is significantly smaller. Despite the fact that a significant observable shift emerges in the location of the solution curves,  using such a range does not invalidate inequality (\ref{condtur}) for $\rho$. Figure \ref{realcomplex} (b) shows the variation of parameter $d$ using the values indicated on each curve. It can be noted that as $d$ increases significantly, the number of intercepts on $\beta$ axis reduces from two to one and eventually to zero. The fact that condition (\ref{condtur}) on the domain size forbids the existence of a region in the admissible parameter space that corresponds to complex conjugate pair of $\sigma_{1,2}$ with positive real part, is related to the behaviour of the curve satisfying (\ref{polyn1}). It is the location of the curve associated to the solution of (\ref{polyn1}), that either admits or restricts the curve satisfying (\ref{polyn2}) on the admissible parameter space. This relation is in agreement and is a numerical demonstration of the conditions for diffusion-driven instability presented in \cite{paper17, paper38, book1, book7}. 
 \begin{figure}[H]
 \centering
 \small
  \begin{subfigure}[h]{.495\textwidth}
    \centering
 \includegraphics[width=\textwidth]{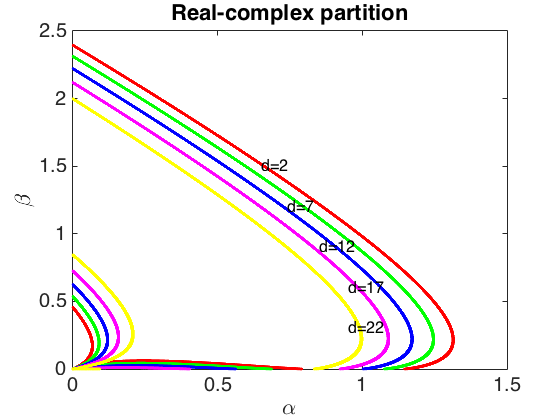}
 \caption{Partitioning curves satisfying (\ref{part1}) and\\ condition (\ref{comp4}) of Theorem \ref{theorem2}, that partitions\\ the region that corresponds to real $\sigma_{1,2}$\\ from that, which corresponds to complex\\ conjugate pair of $\sigma_{1,2}$}
 \label{realcomplexb}
 \end{subfigure}
 \begin{subfigure}[h]{.495\textwidth}
 \centering
    \includegraphics[width=\textwidth]{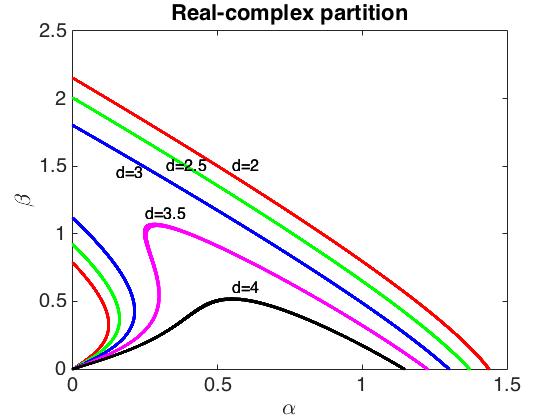}
  \caption{Partitioning curves satisfying (\ref{part1}) and\\ condition (\ref{condtur}) of Theorem \ref{theorem3}, that partitions\\ the region corresponding to real pair of $\sigma_{1,2}$\\ from that, which corresponds to complex\\ conjugate pair of $\sigma_{1,2}$}
  \label{realcomplexa}
 \end{subfigure}
 \caption{The effect of varying $d$ on the solution curves satisfying (\ref{part1}), where $\rho$ is used according to conditions (\ref{comp4}) and (\ref{condtur}) of Theorems \ref{theorem2} and \ref{theorem3} respectively.}
  \label{realcomplex}
\end{figure}
A trial and error method is used to identify which side of the partitioning curves in Figure \ref{realcomplex}, corresponds to complex $\sigma_{1,2}$, where the shift of such regions in the parameter space subject to the same respective variation of $d$ is presented in Figure \ref{complex}. 
 \begin{figure}[H]
 \centering
 \small
  \begin{subfigure}[h]{.495\textwidth}
    \centering
 \includegraphics[width=\textwidth]{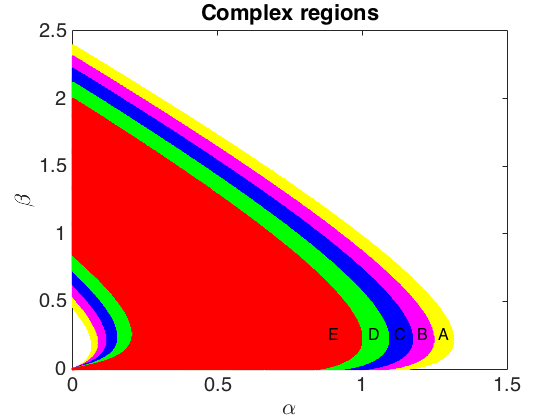}
 \caption{Shift of regions corresponding to \\complex $\sigma_{1,2}$ and subject to condition\\ (\ref{comp4}) of Theorem \ref{theorem2}}
 \label{complexb}
 \end{subfigure}
 \begin{subfigure}[h]{.495\textwidth}
 \centering
    \includegraphics[width=\textwidth]{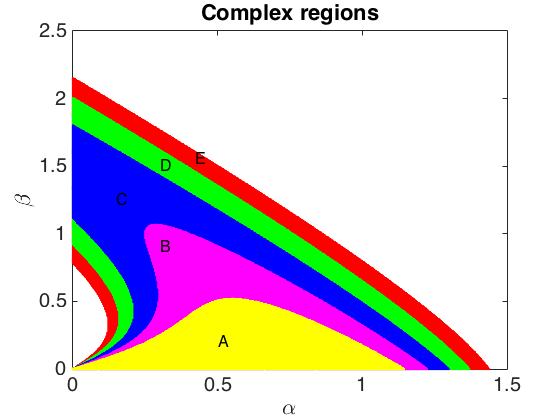}
  \caption{Shift of regions corresponding to \\complex $\sigma_{1,2}$ and subject to condition\\ (\ref{condtur}) of Theorem \ref{theorem3}}
  \label{complexa}
 \end{subfigure}
 \caption{The shift in parameter spaces corresponding to complex $\sigma_{1,2}$ as a consequence of varying $d$.}
  \label{complex}
\end{figure}
Investigating the regions corresponding to complex values for $\sigma_{1,2}$ in Figure \ref{complex}, using the solution of (\ref{polyn2}), it was found that a sub-partition only exists if the value of $\rho$ satisfies condition (\ref{comp4}), with respect to the values of $d$ and $\gamma$. This is a numerical verification of Theorem \ref{theorem2}. If the values of $d$ and/or $\gamma$ are changed such that $\rho$ no longer satisfies condition (\ref{comp4}), it causes to vanish the existence of a sub-partition, within the region corresponding to complex eigenvalues $\sigma_{1,2}$, which is in agreement with Theorem \ref{theorem3}. Under such choices of $d$ and $\gamma$, the region of the parameter space corresponding to complex $\sigma_{1,2}$ has no stability partition, therefore, everywhere on this region the real part of $\sigma_{1,2}$ is negative. If parameters $\alpha$ and $\beta$ are fed into system (\ref{polarsystem}), then the dynamics are expected to exhibit global temporal stability. Figure \ref{complexs} shows the regions on the admissible parameter spaces corresponding to complex $\sigma_{1,2}\in\mathbb{C}\backslash\mathbb{R}$ with negative real part. It can be noted that Figure \ref{complexs} (b) portrays exactly the same spaces as shown in Figure \ref{complex} (b), which is a further verification of Theorem \ref{theorem3}, namely, when $\rho$ satisfies condition (\ref{condtur}), then for no choice of $\alpha,\beta \in \mathbb{R}_+$ the complex eigenvalue $\sigma_{1,2}$ can have positive real part.
It further means that if $\rho$ is chosen such that it satisfies condition (\ref{condtur}), then system (\ref{polarsystem}) is guranteed to exhibit global temporal stability in the dynamics, with the only possibility of spatial periodic behaviour. The consequence of this relationship is to state that if the value of $\rho$ is bounded above by the right hand-side of condition (\ref{condtur}), then the diffusion-driven instability will remain restricted to Turing type only, hence any pattern obtained under condition (\ref{condtur}) on $\rho$ is expected to be spatially periodic pattern, with a global temporally stable behaviour. In this case we can only obtain a pattern of spots or stripes, with spatial periodicity. 
\begin{figure}[H]
 \centering
 \small
  \begin{subfigure}[h]{.495\textwidth}
    \centering
 \includegraphics[width=\textwidth]{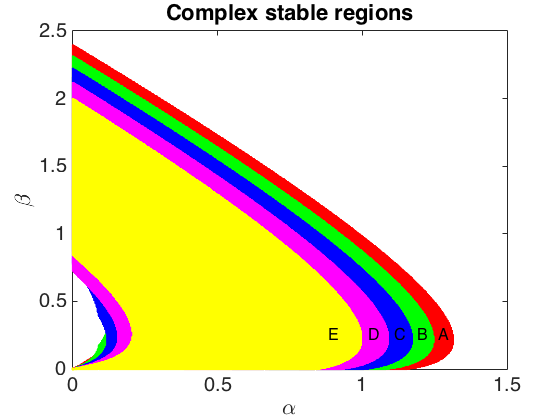}
 \caption{Shift of regions corresponding to \\complex $\sigma_{1,2}$ with negative real part\\ and subject to condition (\ref{comp4}) of \\Theorem \ref{theorem2}}
 \label{complexsb}
 \end{subfigure}
 \begin{subfigure}[h]{.495\textwidth}
 \centering
    \includegraphics[width=\textwidth]{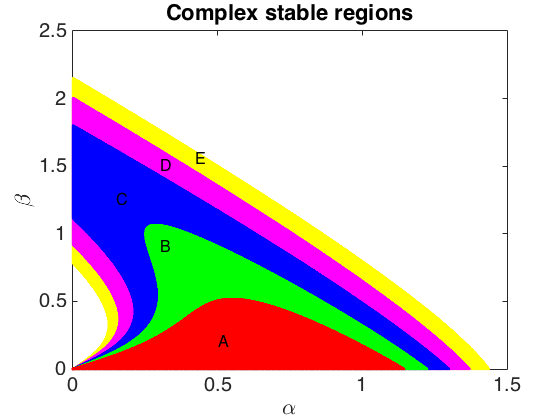}
  \caption{Shift of regions corresponding to \\complex $\sigma_{1,2}$ with negative real parts \\and subject to condition (\ref{condtur}) of \\Theorem \ref{theorem3}}
  \label{complexsa}
 \end{subfigure}
 \caption{The shift in parameter spaces corresponding to complex $\sigma_{1,2}$ with negative real parts as a consequence of varying $d$.}
  \label{complexs}
\end{figure}
If a condition on $\rho$ is set so that it is large enough to exceed the value on the right hand-side of condition (\ref{condtur}), i.e. $\rho$ satisfying (\ref{comp4}), only then a sub-partition can emerge within the admissible parameter space corresponding to complex pair of $\sigma_{1,2}$. This can be observed by comparing Figure \ref{complexs} (a) with Figure \ref{complex} (a), in particular the deficiency of regions is clearly visible on the inner stripes of the shifted spaces near the origin in Figure \ref{complexs} (a). Figure \ref{comufig} (b) shows the emergence of these curves that partition the region corresponding to complex $\sigma_{1,2}$. Recalling that if a sub-partition in the regions indicated by Figure \ref{complex} exists, then the corresponding partitioning curves must satisfy (\ref{part2}), which resemble the values of the parameter space that causes the real part of $\sigma_{1,2}$ to become zero when it is a pair of complex conjugate values. Therefore, on these curves the uniform steady state $(u_s,v_s)$ undergoes transcritical bifurcation. These are also the curves on which the real part of $\sigma_{1,2}\in \mathbb{C}\backslash \mathbb{R}$ changes sign, which means to one side of these curves a region in the admissible parameter space exists that corresponds to $\sigma_{1,2}$ to be a complex conjugate pair but with positive real part. Parameter values of such kind result in the uniform steady state $(u_s,v_s)$ to become periodic in time. This entails that the values of $\alpha$ and $\beta$ from this region causes system (\ref{polarsystem}) to exhibit Hopf bifurcation. Figure \ref{comufig} (a) shows a shift in the region of the parameter spaces that corresponds to Hopf bifurcation for the same variation in the value of $d$ as used in Figure \ref{complex}. It is also worth noting that with increasing values of $d$, the parameter spaces corresponding to Hopf bifurcation gradually decrease. This is in agreement with the mathematical reasoning behind Theorem \ref{theorem2}, because as the value of $d$ is increased, one gets closer to the violation of the necessary condition (\ref{comp4}) for the existence of regions for Hopf bifurcation. Figure \ref{comufig} (b) shows a more gradual extinction of the sub-partitioning curves within the region corresponding to complex $\sigma_{1,2}$. For eight different values of $d$, with $\rho$ still satisfying (\ref{comp4}), the solution of (\ref{polyn1}) and  (\ref{polyn2}) was plotted near the origin to closely observe the interaction and a shift of the location. Where they intersect each other, it was found that for small values of $d$, the intersection between the curves (\ref{polyn1}) and (\ref{polyn2}) occurs such that a significant part of the curve satisfying (\ref{polyn2}) forms a sub-partition within the region of the admissible parameter space that corresponds to complex eigenvalues. However, as the value of $d$ is gradually increased, so did the location of the intersection moved, reducing the sub-partition formed by the solution of (\ref{polyn2}). 
\begin{figure}[H]
 \centering
 \small
  \begin{subfigure}[h]{.495\textwidth}
    \centering
 \includegraphics[width=\textwidth]{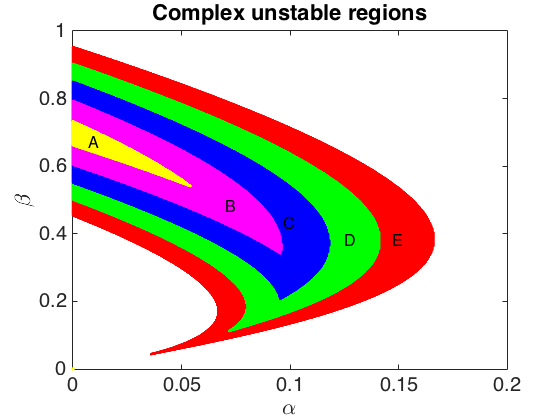}
 \caption{Shift of regions corresponding to \\complex $\sigma_{1,2}$ with positive real part and \\$\rho$ restricted to condition (\ref{comp4}) of Theorem \ref{theorem2}}
 \label{compub}
 \end{subfigure}
 \begin{subfigure}[h]{.495\textwidth}
 \centering
    \includegraphics[width=\textwidth]{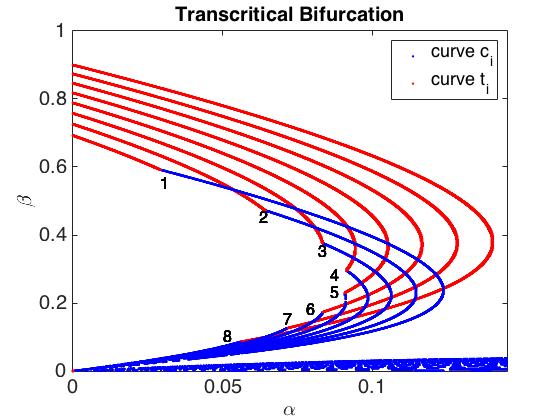}
  \caption{Shift in the location of red curves $t_i$ for eight different values of $d$, on which the values of $\sigma_{1,2}$ are pure imaginary, and $\rho$ is restricted to condition (\ref{comp4}) of Theorem \ref{theorem2}}
  \label{tranub}
 \end{subfigure}
 \caption{Regions in (a) correspond to Hopf bifurcation and in (b) curves $t_i$ correspond to transcritical bifurcation. When $\rho$ satisfies condition (\ref{condtur}), then no values of $\alpha$ and $\beta$ give rise to temporal instability in system (\ref{polarsystem}).}
  \label{comufig}
\end{figure}
Curves $c_i$ for $i\in\{1,2,3,4,5,6,7,8\}$ (blue colour) in Figure \ref{comufig} (b) denote the solution satisfying (\ref{polyn1}). On curves $c_i$ the eigenvalues $\sigma_{1,2}$ are repeated positive real roots, therefore, parameter values on these curves correspond to Turing type instability, whereas curves $t_i$ (in red colour) indicate the sub-partition formed by the solution of (\ref{polyn2}) within the region corresponding to complex eigenvalues $\sigma_{1,2}$. On curves $t_i$ the uniform steady state $(u_s,v_s)$ is expected to undergo transcritical bifurcation. In Figure \ref{comufig} (b) curves $t_i$ and $c_i$ for $i=8$ correspond to the smallest value of $d=2$, which is gradually increased reaching a maximum value of $d=22$ that corresponds to curves $c_1$ and $t_1$.

The whole of the admissible parameter space consisting of the top right quadrant of $\mathbb{R}_+^2$ plane is partitioned by the union of the spaces presented in Figure \ref{complex} and those corresponding to eigenvalues that are real. Figure \ref{realfig} shows, under both conditions (\ref{comp4}) and (\ref{condtur}) on $\rho$, the shift of parameter spaces that correspond to real eigenvalues for different values of $d$.  
 \begin{figure}[ht]
 \centering
 \small
  \begin{subfigure}[h]{.495\textwidth}
    \centering
 \includegraphics[width=\textwidth]{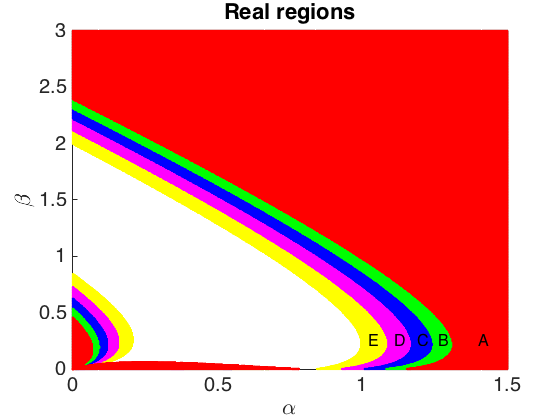}
 \caption{Shift of regions corresponding to \\real $\sigma_{1,2}$ and subject to condition\\ (\ref{comp4}) of Theorem \ref{theorem2}}
 \label{realb}
 \end{subfigure}
 \begin{subfigure}[h]{.495\textwidth}
 \centering
    \includegraphics[width=\textwidth]{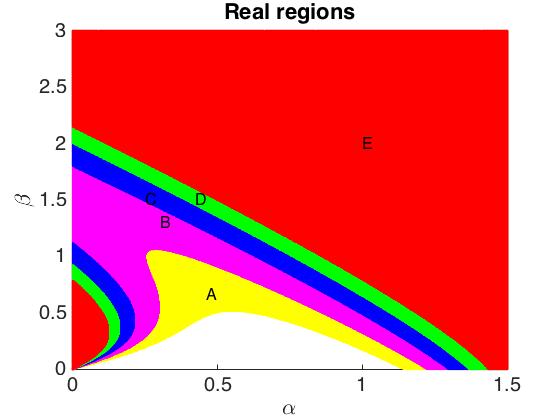}
  \caption{Shift of regions corresponding to \\real $\sigma_{1,2}$ and subject to condition\\ (\ref{condtur}) of Theorem \ref{theorem3}}
  \label{reala}
 \end{subfigure}
 \caption{The shift in parameter spaces corresponding to real $\sigma_{1,2}$ as a consequence of varying $d$.}
  \label{realfig}
\end{figure}
 The outer partitioning curves of the regions corresponding to complex eigenvalues in Figure \ref{realcomplex}, indicate the choice of $\alpha$ and $\beta$ for which the eigenvalues $\sigma_{1,2}$ is a pair of real repeated negative values, therefore, parameter spaces bounded by these curves corresponds to a pair of distinct negative real values. Parameter choice from these regions corresponds to a global spatio-temporally stable behaviour of the dynamic of system (\ref{polarsystem}). Figure \ref{realsfig} shows the shift of these spatio-temporal stable regions on the admissible parameter space. Any choice of $\alpha$ and $\beta$ from these regions will result in the dynamics of system (\ref{polarsystem}) to exhibit global stability in space as well as in time, which means that if system (\ref{polarsystem}) is perturbed in the neighbourhood of the uniform steady state $(u_s,v_s)$ using parameters from these regions, the dynamics will guarantee to return back to the uniform steady state $(u_s,v_s)$. It can be clearly observed that under both conditions (\ref{comp4}) and (\ref{condtur}) on $\rho$, regions corresponding to a pair of real negative eigenvalues form a proper subset of those spaces that corresponds to a pair of real values irrespective of sign. Figure \ref{realsfig} shows under both conditions (\ref{comp4}) and (\ref{condtur}) on $\rho$ the shift in the spaces that correspond to global spatio-temporally stable behaviour of the dynamics of system (\ref{polarsystem}). Comparing Figure \ref{realfig} and Figure \ref{realsfig} it can be clearly observed that under both conditions on $\rho$ the regions of spaces corresponding to negative real values are deformed versions of those that correspond to arbitrary real values.  
 \begin{figure}[H]
 \centering
 \small
  \begin{subfigure}[h]{.495\textwidth}
    \centering
 \includegraphics[width=\textwidth]{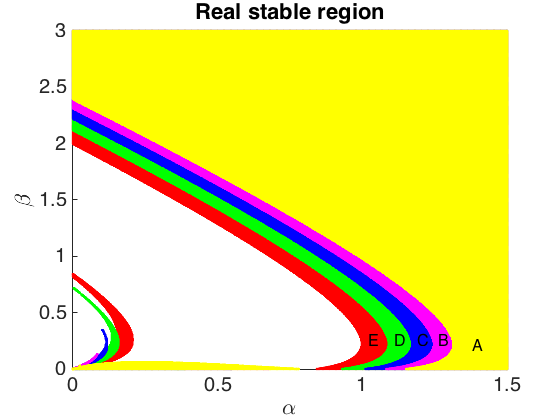}
 \caption{Shift of regions corresponding to \\a pair of negative distinct \\real $\sigma_{1,2}$ and subject to condition\\ (\ref{comp4}) of Theorem \ref{theorem2}}
 \label{realsb}
 \end{subfigure}
 \begin{subfigure}[h]{.495\textwidth}
 \centering
    \includegraphics[width=\textwidth]{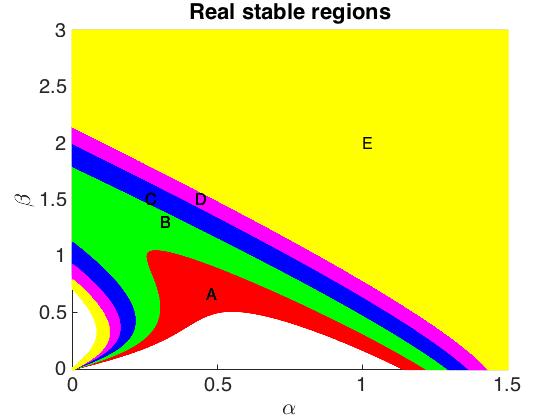}
  \caption{Shift of regions corresponding to \\a pair of negative distinct \\real $\sigma_{1,2}$ and subject to condition\\ (\ref{condtur}) of Theorem \ref{theorem3}}
  \label{realsa}
 \end{subfigure}
 \caption{The shift in parameter spaces corresponding to negative real distinct $\sigma_{1,2}$ as a consequence of varying $d$.}
  \label{realsfig}
\end{figure}
The remaining spaces to analyse are those corresponding to diffusion-driven instability of Turing type under conditions (\ref{comp4}) and (\ref{condtur}) on $\rho$, when $\sigma_{1,2}$ are a pair of real values with at least one of them positive. This region corresponds to Turing type instability and under both conditions on $\rho$ these regions exist. It can be noted that near the origin of the admissible parameter space in Figure \ref{realcomplex}, for each value of $d$ the small curves starting at the origin $(\alpha,\beta)=(0,0)$ and curving back to intercept the $\beta$ axis, are the curves on which the eigenvalues $\sigma_{1,2}$ are repeated positive real roots, therefore these curves correspond to diffusion-driven instability of Turing type.
We know that the diffusion-driven instability can also happen, when either $\sigma_1$ or $\sigma_2$ are positive real. Figure \ref{realufig} shows the shift of those regions corresponding to Turing type instability and it can be observed that as $d$ increases, the region in the parameter space enlarges. In Figures \ref{complexs} and \ref{realufig}, all the points specific to a certain colour on the parameter plane are denoted by an alphabetic letter. This is for the purpose to be able to cross reference using set notation to a specific region when summarising the results in Table \ref{table1}.
 \begin{figure}[H]
 \centering
 \small
  \begin{subfigure}[h]{.495\textwidth}
    \centering
 \includegraphics[width=\textwidth]{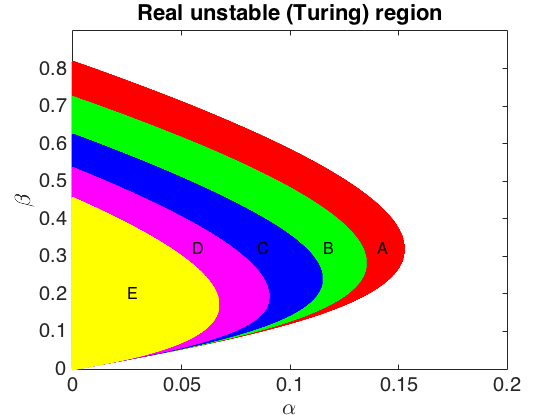}
 \caption{Shift of regions corresponding to Turing \\instability where at least one eigenvalue $\sigma_{1,2}$ \\is real positive and subject to condition \\(\ref{comp4}) of Theorem \ref{theorem2}}
 \label{realub}
 \end{subfigure}
 \begin{subfigure}[h]{.495\textwidth}
 \centering
    \includegraphics[width=\textwidth]{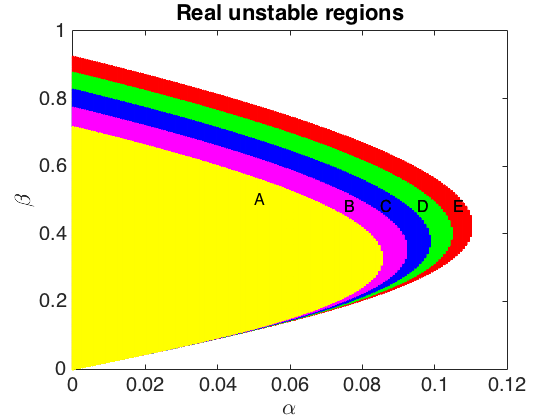}
  \caption{Shift of regions corresponding to Turing \\instability where at least one eigenvalue $\sigma_{1,2}$ \\is real positive and subject to condition \\(\ref{condtur}) of Theorem \ref{theorem3}}
  \label{realua}
 \end{subfigure}
 \caption{The shift in parameter spaces corresponding to at least one positive real eigenvalue $\sigma_{1,2}$ as a consequence of varying $d$.}
  \label{realufig}
\end{figure}
\begin{table}
\centering
\small
\tabcolsep=0.3cm
\noindent\adjustbox{max width=\textwidth}{
\begin{tabular}{|c|c|c|c|c|c|c|c|}\cline{3-8}\hline
\multicolumn{3}{|c|}{Stability of USS $(u_s,v_s)$}&\multicolumn{2}{|c|}{Stable regions}&\multicolumn{3}{|c|}{Unstable regions}\\\hline
\multicolumn{3}{|c|}{Types of USS $(u_s,v_s)$}&Node&Spiral&Turing-instability&Hopf bifurcation& Transcritical bifurcation\\\cline{1-8}
\multicolumn{3}{|c|}{Figure index}&Figure \ref{realsfig} (a)&Figure \ref{complexs} (a)&Figure \ref{realufig} (a)&Figure \ref{comufig} (a)&Figure \ref{comufig} (b)\\\hline
\multirow{5}{*}{\begin{turn}{-90}Theorem \ref{theorem2} \end{turn}}&\multirow{5}{*}{\begin{turn}{-90} $\rho$ satisfying \begin{turn}{90}(\ref{comp4})\end{turn} \end{turn} }&\diaghead{\theadfont{\normalsize} Type of (SS) }{$(d,\gamma,\rho,n)$}{$\sigma_{1,2}$}&$0>\sigma_{1,2}\in\mathbb{R}$&$\sigma_{1,2}\in\mathbb{C}\backslash\mathbb{R}, \text{Re}(\sigma)<0$&$0<\sigma_{1}\in\mathbb{R}$ or $0<\sigma_{2}\in\mathbb{R}$&$\sigma_{1,2}\in\mathbb{C}\backslash\mathbb{R}\text{, Re}(\sigma_{1,2})>0$&$\sigma_{1,2}\in\mathbb{C}\backslash\mathbb{R},\text{ Re}(\sigma_{1,2})=0$\\\cline{3-8}
&&$(2.0,1,35,1.7)$&$A$&$A \cup B \cup C \cup D \cup E $&$E$&$A \cup B \cup C \cup D \cup E$&$t_8$\\\cline{3-8}
&&$(7.0,1,35,1.7)$&$A \cup B$&$B \cup C \cup D \cup E$&$E \cup D$&$A \cup B \cup C \cup D$&$t_6$\\\cline{3-8}
&&$(12,1,35,1.7)$&$A \cup B \cup C$&$C \cup D \cup E$&$E \cup D \cup C$&$A \cup B \cup C$&$t_3$\\\cline{3-8}
&&$(17,1,35,1.7)$&$A \cup B \cup C \cup D$&$D \cup E$&$E \cup D \cup C \cup B$&$A \cup B$&$t_2$\\\cline{3-8}
&&$(22,1,35,1.7)$&$A \cup B \cup C \cup D \cup E$&$E$&$E \cup D \cup C \cup B \cup A$&$A$&$t_1$\\\hline
\multicolumn{3}{|c|}{Figure index}&Figure \ref{realsfig} (b)&Figure \ref{complexs} (b)&Figure \ref{realufig} (b)&Figure \ref{comufig} (b)&Figure \ref{complexs} (b)\\\hline
\multirow{5}{*}{\begin{turn}{-90}Theorem \ref{theorem3} \end{turn}}&\multirow{5}{*}{\begin{turn}{-90} $\rho$ satisfying \begin{turn}{90}(\ref{condtur})\end{turn} \end{turn} }&\diaghead{\theadfont{\normalsize} Type of (SS) }{$(d,\gamma,\rho,n)$}{$\sigma_{1,2}$}&$0>\sigma_{1,2}\in\mathbb{R}$&$\sigma_{1,2}\in\mathbb{C}\backslash\mathbb{R}, \text{Re}(\sigma)<0$&$0<\sigma_{1}\in\mathbb{R}$ or $0<\sigma_{2}\in\mathbb{R}$&$\sigma_{1,2}\in\mathbb{C}\backslash\mathbb{R}\text{, Re}(\sigma_{1,2})>0$&$\sigma_{1,2}\in\mathbb{C}\backslash\mathbb{R},\text{ Re}(\sigma_{1,2})=0$\\\cline{3-8}
&&$(2.0,1,10,1.7)$&$E$&$A \cup B \cup C \cup D \cup E$&$A$&$\emptyset$&$\emptyset$\\\cline{3-8}
&&$(2.5,1,10,1.7)$&$E \cup D$&$A \cup B \cup C \cup D$&$A \cup B$&$\emptyset$&$\emptyset$\\\cline{3-8}
&&$(3.0,1,10,1.7)$&$E \cup D \cup C$&$A \cup B \cup C$&$A \cup B \cup C$&$\emptyset$&$\emptyset$\\\cline{3-8}
&&$(3.5,1,10,1.7)$&$E \cup D \cup C \cup B$&$A \cup B$&$A \cup B \cup C \cup D$&$\emptyset$&$\emptyset$\\\cline{3-8}
&&$(4.0,1,10,1.7)$&$E \cup D \cup C \cup B \cup A$&$A$&$A \cup B \cup C \cup D \cup E$&$\emptyset$&$\emptyset$\\\hline
\end{tabular}}
\caption{The summary of the full classification of parameter spaces for system (\ref{polarsystem}) is presented with the associated numerical verification of the predictions made by Theorems \ref{theorem2} and \ref{theorem3}. Furthermore, it is shown how a certain type of bifurcation space shifts according to the variation of parameter $d$.}
\label{table1}
\end{table}
\subsubsection{\bf Remark}\label{remark2}
{\it The first non-trivial mode $k=1$ corresponding to the pair $(j=1,j=2)$ in the infinite series of the eigenfunctions is used for all the computations of the partitioning curves and the corresponding parameter spaces. It is worth noting that higher modes such as $k=2,3,...$ for the expressions of the eigenvalues $\eta_{n,k}$ will only cause a scaled shift in the location of the partitioning curves, whilst keeping the validity of conditions (\ref{comp4}) and (\ref{condtur}) intact. This is because the choice of higher modes $k$ enforces a corresponding change in the value of the radius $\rho$ in order for inequalities corresponding to conditions (\ref{comp4}) and (\ref{condtur}) to hold. In other words, given that conditions (\ref{comp4}) and (\ref{condtur}) are satisfied the dynamics are guaranteed to exhibit the proposed behaviour independent of the choice of the mode $k$ and the order of the corresponding Bessel's equation $n\in\frac{1}{2}\mathbb{Z}$.}

\section{Finite element solutions of RDS}\label{fem}
To verify numerically the classification of parameter spaces proposed by Theorems \ref{theorem2} and \ref{theorem3}, the reaction-diffusion system (\ref{polarsystem}) is simulated using the finite element method \cite{paper6, paper12, paper15, paper26, paper27, paper41, book2, book4, book5} on a unit disk $\Omega$. Due to the curved boundary of $\Omega$, the triangulation is obtained through an application of an iterative algorithm using a technique called \textit{distmesh} \cite{paper40, paper41}. 
\subsection{Mesh generation using the \textit{distmesh} triangulation algorithm}
The algorithm for \textit{distmesh} was originally developed by Persson and Strang \cite{paper40, paper41}. The algorithm is coded in MATLAB for generating uniform and non-uniform refined meshes on two and three dimensional geometries. The algorithm of \textit{distmesh} uses signed-distance function $d(x,y)$, which is negative inside the discretised domain $\Omega$ and is positive outside the $\partial\Omega$. The construction of distmesh triangulation is an iterative process using a set of two interactive algorithms, one of which controls the displacement of nodes within the domain and the other ensures that the consequences of node displacement does not violate the properties of the Delaunay triangulation \cite{paper43}. For details on how to generate meshes using \textit{distmesh} we refer the interested reader to the paper by Persson and Strang \cite{paper40, paper41, paper44}. 
\subsection{Finite element simulations}
Using the $distmesh$ algorithm, a unit disk $\Omega$ was discretised to obtain a uniform triangulation. Each simulation in this section is performed on a unit disk centred at the origin of cartesian plane, which is discretised by $6327$ triangles consisting of $3257$ nodes. The initial conditions for each simulation are taken to be small random perturbations in the neighbourhood of the steady state $(u_s,v_s)$ of the form \cite{book1, book7, paper38, paper17}
\begin{equation}
\begin{cases}
u_0(x,y)=\alpha+\beta+0.0016\cos(2\pi(x+y))+0.01\sum_{i=1}^8\cos(i\pi x),\\
v_0(x,y)= \frac{\beta}{(\alpha+\beta)^2}+0.0016\cos(2\pi(x+y))+0.01\sum_{i=1}^8\cos(i\pi x).
\end{cases}
\label{initial}
\end{equation} 
The choices of parameters $\alpha$ and $\beta$ are verified from each of the four regions where the dynamics of system (\ref{polarsystem}) exhibit diffusion-driven instability. In all our simulations, the parameters $d$ and $\gamma$ are varied, whilst keeping the value of radius $\rho=1$ fixed, which  allows us to keep constant the well refined number of degrees of freedom for the mesh. It proves computationally expensive to vary $\rho$, therefore in order to satisfy conditions (\ref{comp4}) and (\ref{condtur}), it suffices to change the values of $d$ and $\gamma$ to ensure that a particular condition on the domain size in terms of reaction-diffusion rates is satisfied. The actual numerical values for each simulation are presented in Table \ref{Table2}.
\subsubsection{Global spatio-temporal stability}
Before demonstrating diffusion-driven instability, a pair of parameter values $\alpha,\beta$ from the spatio-temporally stable region indicated in Table \ref{table1} is used to demonstrate, in Figure \ref{stable} (a), how the dynamics overcome small perturbations and reach the uniform steady state $(u_s,v_s)$ without evolving to any pattern. The uniform convergence to the constant steady state $(u_s,v_s)$ is shown in Figure \ref{stable} (b), through the convergence of the discrete $L_2$ norm of the discrete time derivatives of the solutions $u$ and $v$.
\begin{figure}[ht]
 \centering
 \small
  \begin{subfigure}[h]{.495\textwidth}
    \centering
 \includegraphics[width=\textwidth]{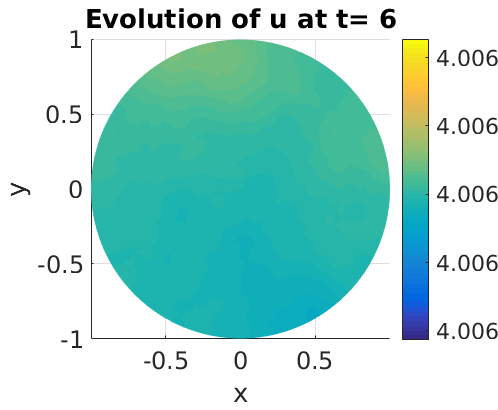}
 \caption{No pattern is evolved when parameters\\ $\alpha$ and $\beta$ are chosen outside Turing space\\ under condition (\ref{condtur}) on the radius $\rho$.}
 \label{evolu1}
 \end{subfigure}
 \begin{subfigure}[h]{.495\textwidth}
 \centering
    \includegraphics[width=\textwidth]{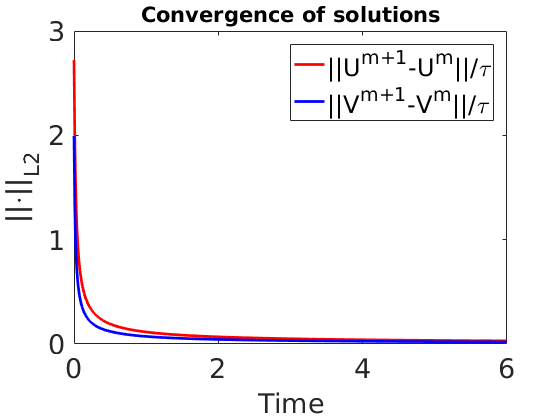}
  \caption{Uniform convergence of the discrete $L_2$ norm of the discrete time derivatives of the solutions $u$ and $v$.}
  \label{error1}
 \end{subfigure}
 \caption{When $\rho$ is bounded by a combination of $d$ and $\gamma$ (as shown in Table \ref{Table2}) according to condition (\ref{condtur}), then no choice of $(\alpha,\beta)$ outside Turing-space can trigger instability in the dynamics, hence no pattern emerges.}
  \label{stable}
\end{figure}
\subsubsection{Turing instability when the radius is small}
  Figure \ref{unstable1} (a) shows the evolution of a spatial pattern as a consequence of choosing $(\alpha,\beta)$ from the Turing region under condition (\ref{condtur}) indicated in Figure \ref{realufig} (b). 
  Depending on the initial conditions and the mode of the eigenfunctions, the spatially periodic pattern provided by parameter spaces in Figure \ref{realufig} (b) is expected to be a combination of radial and angular stripes or spots. Once the initial pattern is formed by the evolution of the dynamics, then the system is expected to uniformly converge to a Turing-type steady state, which means the initially evolved spatial pattern becomes temporally invariant as time grows. The simulation of Figure \ref{unstable1} was executed for long enough time such that the discrete time derivative of solutions $u$ and $v$ decaying to a threshold of $5\times 10^{-4}$ in the discrete $L_2$ norm. Figure \ref{unstable1} (b) demonstrates the behaviour of the discrete time derivatives of both species for the entire period of simulation time until the threshold was reached. It is observed that after the initial Turing-type instability, the evolution of the system uniformly converges to a spatially patterned steady state   
\begin{figure}[ht] 
 \centering
 \small
  \begin{subfigure}[h]{.495\textwidth}
    \centering
 \includegraphics[width=\textwidth]{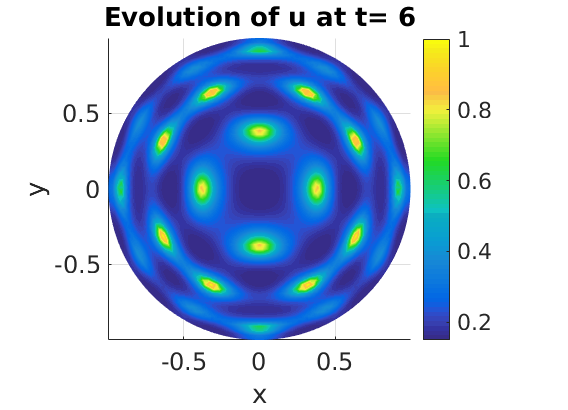}
 \caption{Pattern evolution is restricted to\\ spatial periodicity for $\alpha$ and $\beta$ in the Turing \\space under condition (\ref{condtur}) on the radius $\rho$.}
 \label{evolu2}
 \end{subfigure}
 \begin{subfigure}[h]{.495\textwidth}
 \centering
    \includegraphics[width=\textwidth]{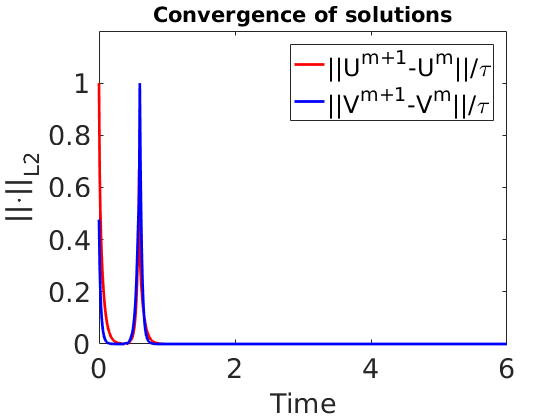}
  \caption{Instability of the discrete $L_2$ norm of the discrete time derivatives of the solutions $u$ and $v$.}
  \label{error2}
 \end{subfigure}
 \caption{When $\rho$ is bounded by a combination of $d$ and $\gamma$ (as shown in Table \ref{Table2}) according to condition (\ref{condtur}), then the only admissible pattern is a spatially periodic pattern for $(\alpha,\beta)$ from the Turing-space shown in Figure \ref{realufig} (b).}
  \label{unstable1}
\end{figure}
\subsubsection{Turing instability when the radius is large} 
Figure \ref{unstable2} presents a series of three snapshots to show how the spatially periodic pattern is evolved to a Turing type steady state, when parameters $\alpha$ and $\beta$ are chosen from Turing region and $\rho$ satisfying condition (\ref{comp4}), with respect to $d$ and $\gamma$. For simulations in Figure \ref{unstable2}, parameters $\alpha$ and $\beta$ are chosen from regions presented in Figure \ref{realufig} (a); the dynamics within these regions evolve to a spatial pattern with global temporal stability.  It can be noted from Figure \ref{unstable2} (d), that after the initial pattern is formed, the system is uniformly converging to the spatially periodic steady state. Turing instability is a domain independent phenomena, and only depends on the choice of parameters and the initial conditions. At $t=6$ the required threshold of $5\times 10^{-4}$ on the discrete $L_2$ norm of the discrete time derivatives of the solutions $u$ and $v$ is reached, which can be observed in Figure \ref{unstable2} (d).
The remaining two unstable regions in the admissible parameter space presented in Figure \ref{comufig} correspond to spatio-temporal periodicity.
\begin{figure}[H] 
 \centering
 \small
  \begin{subfigure}[h]{.495\textwidth}
    \centering
 \includegraphics[width=\textwidth]{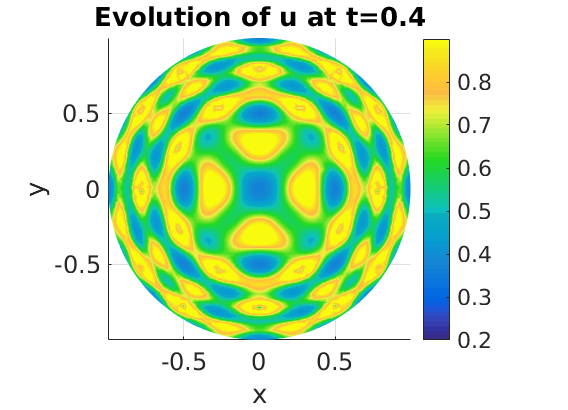}
 \caption{Initial stage of evolving spatially periodic \\ pattern when, $\alpha$ and $\beta$ are chosen from \\Turing space under condition (\ref{comp4}) on \\the radius $\rho$.}
 \label{evolu3a}
 \end{subfigure}
 \begin{subfigure}[h]{.495\textwidth}
 \centering
    \includegraphics[width=\textwidth]{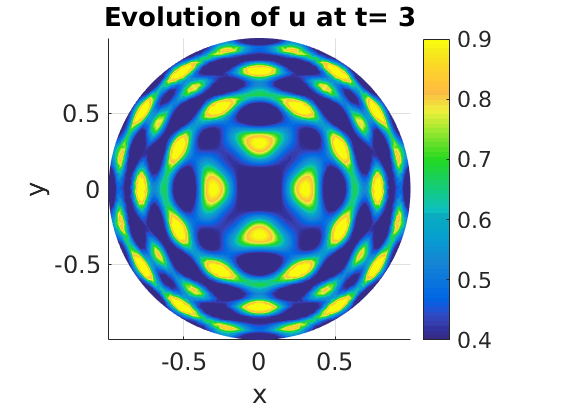}
  \caption{Spatially periodic pattern at $t=3$ is as expected converging to the Turing type steady state without allowing the initial pattern to be deformed}
  \label{evolv3b}
 \end{subfigure}
 \begin{subfigure}[h]{.495\textwidth}
    \centering
 \includegraphics[width=\textwidth]{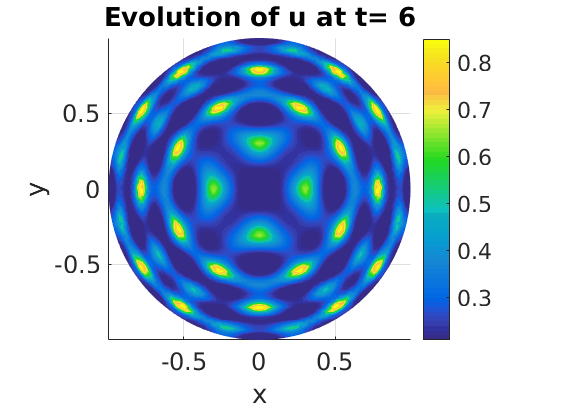}
 \caption{Turing type steady state is reached at \\$t=6$, with a threshold in the discrete \\ $L_2$ norm of $5\times10^{-4}$ satisfied for the\\ discrete time derivatives of the solutions.}
 \label{evolu3c}
 \end{subfigure}
 \begin{subfigure}[h]{.495\textwidth}
 \centering
    \includegraphics[width=\textwidth]{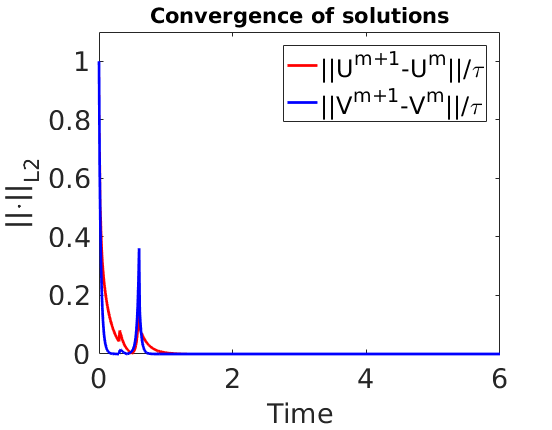}
  \caption{Instability and convergence is shown through the behaviour of the discrete time derivatives of the solutions $u$ and $v$.}
  \label{error3}
 \end{subfigure}
 \caption{When $\rho$ is large with respect to the combination of $d$ and $\gamma$ (as shown in Table \ref{Table2}) according to condition (\ref{comp4}), then the dynamics admit spatial diffusion-driven instability for $(\alpha,\beta)$ from the Turing-space.}
  \label{unstable2}
\end{figure}
\subsubsection{Hopf bifurcation}
 Regions presented in Figure \ref{comufig} (a) are those corresponding to a complex conjugate pair of eigenvalues with non-zero positive real part. These regions emerge in the admissible parameter space under condition (\ref{comp4}) on $\rho$. Choosing parameters from regions in Figure \ref{comufig} (a) admits temporal periodic behaviour in the dynamics of system (\ref{polarsystem}) as shown in Figure \ref{hopf}. The initial pattern in Figure \ref{hopf} (a) is in fact achieved earlier than at $t=1$, therefore, the temporal gap between the initial and second pattern in Figure \ref{hopf} (b) is relatively smaller than the temporal gaps that exist between the second, third and fourth temporal periods. It is worth noting that the temporal period between the successive transitional temporal instabilities from one type of spatial pattern to another grows larger with time. The initial pattern is obtained at around $t\approx1$, which becomes unstable during the transition to the second temporal period at $t\approx5$ in Figure \ref{hopf} (b). At $t\approx8$ the system undergoes a third period of instability and reaches a different spatial pattern at $t\approx12$ shown in Figure \ref{hopf} (c).
 The fourth period of temporal instability is reached at $t\approx 20$, which converges to the fourth temporally-local but spatially periodic steady state at $t\approx28$ presented in Figure \ref{hopf} (d). It follows that when parameters are chosen from the Hopf bifurcation region then the temporal gaps in the dynamics of system (\ref{polarsystem}) between successive transitional instabilities from one spatial pattern to another is approximately doubled as time grows. It is speculated that the temporal period-doubling behaviour is connected to the analogy of unstable spiral behaviour in the theory of ordinary differential equations \cite{book12}. If the eigenvalues of a dynamical system modelled by a set of ordinary differential equations is a complex number with positive real part, then the cycles of the corresponding unstable spiral grow larger as time grows. The long-term evolution of temporal instability depends on the magnitude of the real part of $\sigma_{1,2}$. If parameters $(\alpha,\beta)$ are chosen such that the trace of the stability matrix (\ref{vect2}) is large and yet the discriminant is negative, then the dynamics exhibit long-term temporal periodicity, which means that the frequency of temporal cycles will become smaller. A decaying frequency in temporal cycles means that locally on the time axis, the dynamics may exhibit similar behaviour to that of a temporally stable system, therefore, to observe temporal transition from one spatial pattern to another, we solve the model system for a long time. Figure \ref{hopf} (e) shows a visualisation of the transition of such temporal periodicity with a decaying frequency.\\
\begin{figure}[H] 
 \centering
 \small
  \begin{subfigure}[h]{.34\textwidth}
    \centering
 \includegraphics[width=\textwidth]{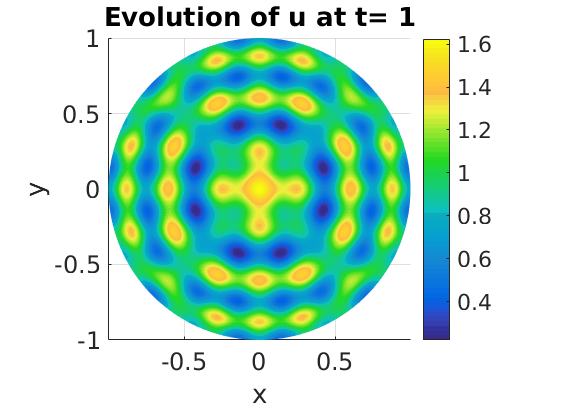}
 \caption{First temporal period\\ evolving the initial spatially\\ periodic pattern at $t=1$\\}
 \label{hopf1}
 \end{subfigure}
 \begin{subfigure}[h]{.32\textwidth}
 \centering
    \includegraphics[width=\textwidth]{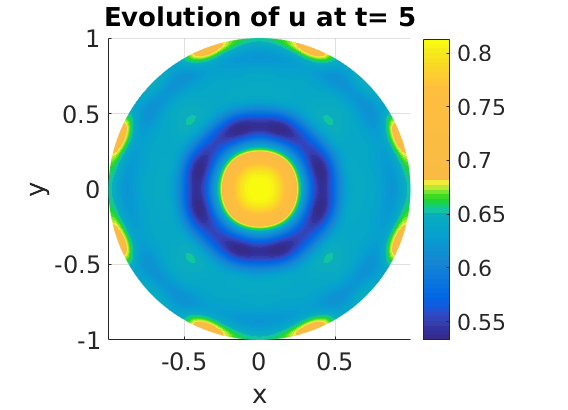}
  \caption{Spatial pattern evolved \\after the first transitional \\instability and during the\\ second temporal period at \\$t=5$}
  \label{hopf2}
 \end{subfigure}
 \begin{subfigure}[h]{.32\textwidth}
    \centering
 \includegraphics[width=\textwidth]{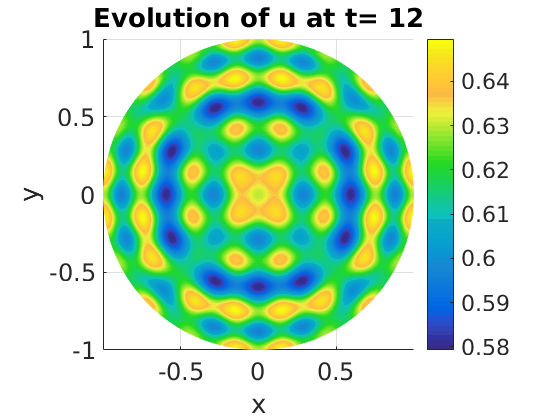}
 \caption{Spatial pattern evolved after  the second transition of temporal instability and during the third temporal period at $t=12$}
 \label{hopf3}
 \end{subfigure}
 \begin{subfigure}[h]{.495\textwidth}
 \centering
    \includegraphics[width=\textwidth]{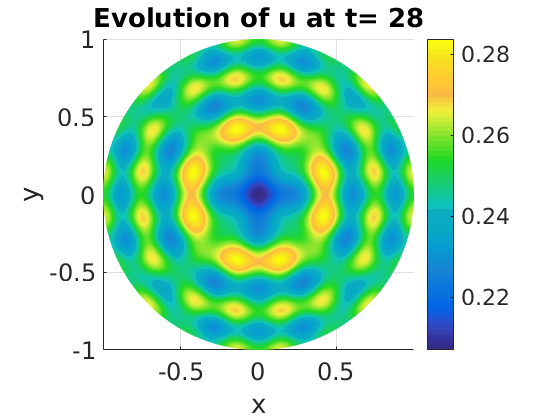}
  \caption{Spatial pattern after the third transition of \\ temporal instability and during the fourth \\temporal period obtained at $t=28$}
  \label{hopf4}
 \end{subfigure}
 \begin{subfigure}[h]{.495\textwidth}
 \centering
    \includegraphics[width=\textwidth]{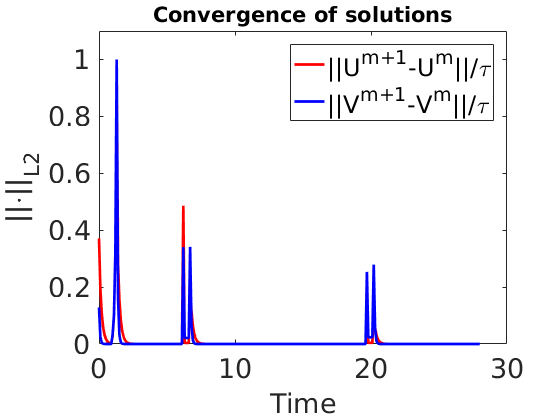}
  \caption{Spatio-temporal periodicity in the dynamics measured in dicrete $L_2$ norm of the successive time-step difference of the solutions $u$ and $v$}
  \label{hopf5}
 \end{subfigure}
 \caption{When $\rho$ is large with respect to the combination of $d$ and $\gamma$ (as shown in Table \ref{Table2}) according to condition (\ref{comp4}), then the dynamics admit spatio-temporal diffusion-driven instability for $(\alpha,\beta)$ from Hopf bifurcation region presented in Figure \ref{comufig} (a).}
  \label{hopf}
\end{figure}
\subsubsection{Transcritical bifurcation}
As stated in Theorem \ref{theorem2}, when $\rho$ satisfies condition (\ref{comp4}) with respect to $d$ and $\gamma$, given that the parameters $(\alpha,\beta)$ chosen from the curves $t_i$ for $i=1,...,8$ as indicated in Figure \ref{comufig} (b), then one may expect the dynamics of system (\ref{polarsystem}) to exhibit spatio-temporal periodic pattern, through a transcritical bifurcation. This kind of behaviour in the dynamics is also known as the limit cycles \cite{book12}. Figure \ref{unstable3}  shows this spatio-temporal periodic behaviour in the evolution of the numerical solution of system (\ref{polarsystem}). This is the case corresponding to parameters that ensure the eigenvalues to be purely imaginary, therefore, it can be observed that the temporal instability occurs with approximately constant periods along the time axis, which verifies the theoretical prediction of the transcritical bifurcation. 
 It is also observed that during the transitional instability from spots in Figure \ref{unstable3} (a) to the angular stripes in Figure \ref{unstable3} (b) the peak of the discrete $L_2$ norm of the discrete time-derivative of the activator $u$ is bigger than that of the inhibitor $v$. However, when the transitional temporal instability occurs to turn the angular stripes in Figure \ref{unstable3} (b) into spots in Figure \ref{unstable3} (c), then the discrete $L_2$ norm of the time-derivative of the inhibitor $v$ exceeds in magnitude than that of the activator $u$. This alternating behaviour can be clearly observed in Figure \ref{unstable3} (e), where in the annotated legend $U$ and $V$ denote the discrete solutions of the activator $u$ and that of the inhibitor $v$.  It can further be understood from Figure \ref{unstable3} (e), that if $(\alpha,\beta)$ are chosen from the curves of the transcritical bifurcation given in Figure \ref{comufig} (b), then the frequency of temporal periods is predicted to remain constant for all times, resulting in a constant interchanging behaviour between different spatial patterns. 
\begin{figure}[H] 
 \centering
 \small
  \begin{subfigure}[h]{.32\textwidth}
    \centering
 \includegraphics[width=\textwidth]{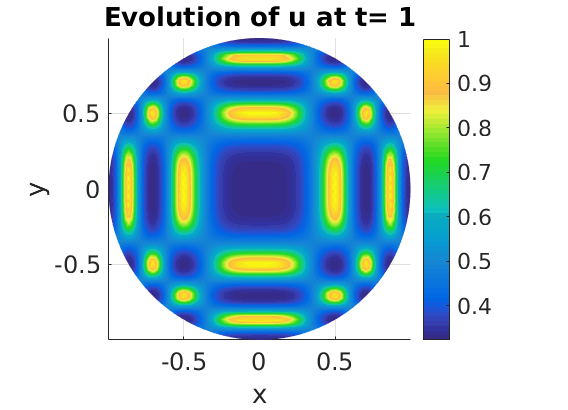}
 \caption{Initial spatially periodic\\ pattern (spots)  obtained at\\ $t=1$ when, $\alpha$ and $\beta$ are\\ chosen from the region of\\ the Hopf bifurcation under \\condition (\ref{comp4}) on the radius\\ $\rho$.}
 \label{evolu4a}
 \end{subfigure}
 \begin{subfigure}[h]{.32\textwidth}
 \centering
    \includegraphics[width=\textwidth]{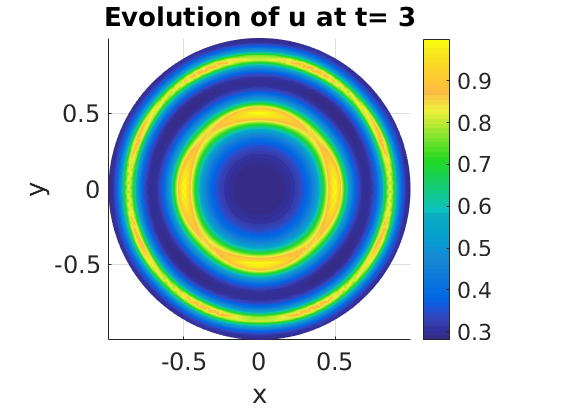}
  \caption{Spatially periodic pattern \\evolves to become a different\\pattern (angular stripes) at\\ $t=3$ for the same choice of\\ parameters as in Figure (a).\\\\}
  \label{evolv4b}
 \end{subfigure}
 \begin{subfigure}[h]{.32\textwidth}
    \centering
 \includegraphics[width=\textwidth]{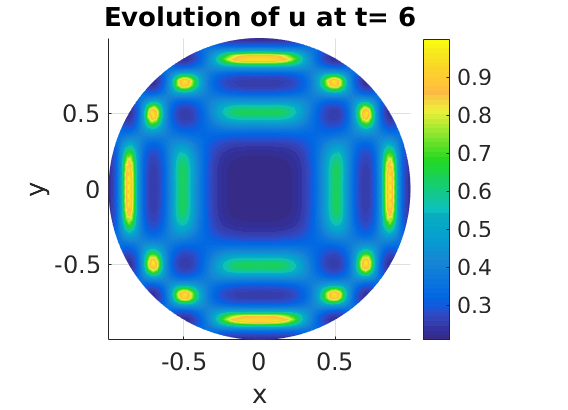}
 \caption{At $t=6$ the angular stripes undergo another period of temporal instability and become spotty as the initial pattern.\\\\}
 \label{evolu4c}
 \end{subfigure}
 \begin{subfigure}[h]{.495\textwidth}
 \centering
    \includegraphics[width=\textwidth]{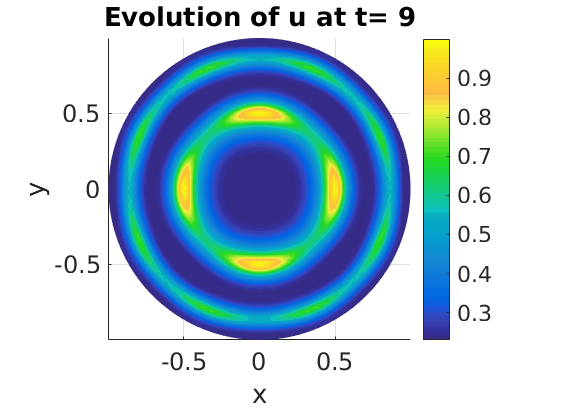}
  \caption{At $t=9$ the pattern of the second \\temporal period emerges again, indicating \\that temporally the dynamics behave in an \\alternating way between the spots and angular\\ stripes.}
  \label{evol4d}
 \end{subfigure}
 \begin{subfigure}[h]{.495\textwidth}
 \centering
    \includegraphics[width=\textwidth]{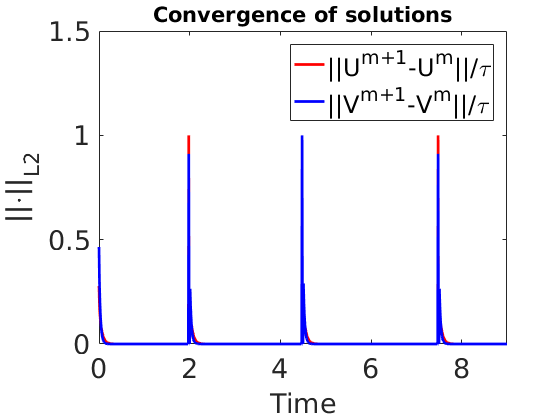}
  \caption{Plot of the discrete $L_2$ norm of the  time-derivative of $u$ and $v$ showing the spatio-temporal behaviour of the solutions for successive time-steps.}
  \label{error4}
 \end{subfigure}
 \caption{When $\rho$ is large with respect to the combination of $d$ and $\gamma$ (as shown in Table \ref{Table2}) according to condition (\ref{comp4}), then the dynamics of (\ref{polarsystem}) can also admit spatio-temporal diffusion-driven instability for $(\alpha,\beta)$ from the transcritical birfucation curves indicated in Figure \ref{comufig} (b).}
  \label{unstable3}
\end{figure}
\begin{table}
\centering
\small
\tabcolsep=0.3cm
\noindent\adjustbox{max width=\textwidth}{
\begin{tabular}{|c |c |c |c |c| c|}
\cline{1-1}
\hline
 Plot index& Figure \ref{stable}  & Figure \ref{unstable1} & Figure \ref{unstable2} &  Figure \ref{hopf} & Figure \ref{unstable3}\\
\hline
\diaghead{\theadfont{\normalsize} Type of (SSSS) }{Parameters}{Instability}&\shortstack{No instability\\ No pattern}& \shortstack{Turing type instability\\Spatial pattern}& \shortstack{Turing type instability\\Spatial pattern}&\shortstack{Hopf bifurcation\\Spatial and temporal pattern} &\shortstack{Transcritical bifurcation\\Spatial and temporal pattern} \\
\hline
$(\alpha,\beta)$ & $(2,2)$ & $(0.1, 0.5)$ & $(0.13, 0.3)$ & $(0.15, 0.4)$ & $(0.05, 0.7)$ \\
\hline
$(n,d, \gamma)$ & $ (2.7,10, 210)$ &$(2.7,10, 210)$& $ (1.7,6, 450)$ & $ (1.7,6,480)$ & $(1.7,6,500)$\\
\hline
Condition on $\Omega$ & (\ref{condtur}) & (\ref{condtur}) &  (\ref{comp4}) &(\ref{comp4}) & (\ref{comp4})\\
\hline
Simulation time & $6$ & $6$ &  $6$ & $28$ & $9$\\
\hline
CPU time (sec) & $784.24$ & $784.56$ &  $784.91$ & $4151.48$ & $1197.66$\\
\hline
\end{tabular}}
\caption{Showing the choice of parameters $(\alpha,\beta)$ for each simulation and the choice of $(n,d,\gamma)$ subject to the relevant condition referred to in third row. Each simulation was run with time-step of $1\times10^{-3}$.} 
\label{Table2}
\end{table}
\newpage
\section{Conclusion}\label{conclusion}
Linear stability theory was applied to a reaction-diffusion system with activator \textit{activator-depleted} reaction kinetics on a two dimensional disk-shape domain. An analytical method was applied to derive explicit expressions for eigenfunctions and the corresponding eigenvalues of the diffusion operator in polar coordinates satisfying homogeneous Neumann boundary conditions. Non-periodic chebyshev grid and periodic Fourier grid was used to discretise a unit disk on which the analytical solutions of the eigenvalue problem were simulated. A colour encoded scheme (HSV) was applied to present the phase plots of the complex valued eigenfunctions. The solution of the eigenvalue problem was used to linearise the reaction-diffusion system for linear stability analysis. An exclusive numerical method, using polynomials was applied to find the solutions of the partitioning curves on the admissible parameter spaces. Analytical methods were used to derive conditions (\ref{comp4}) and (\ref{condtur}) on the radius of a disk-shape domain in the context of bifurcation analysis. It was found that Turing instability occurs independent of the size of the radius, whereas temporal bifurcation in the dynamics are domain dependent, in particular under condition (\ref{comp4}). The relationship between reaction-diffusion rates and the radius of a disk-shape domain was established and analytically proven in Theorems \ref{theorem2} and \ref{theorem3}. The full parameter space was classified with respect to the stability and types of the uniform steady state. Parameter values from all of the regions were tested and the predicted bifurcation in the dynamics was verified using the finite element method. Due to the curved boundary of the domain, \textit{distmesh} was used to obtain the triangulation for simulating the numerical solutions of the system. Spatio-temporal periodicity of Hopf and transcritical types were shown and the corresponding plots of the discrete $L_2$ norms of the time-derivative of the solutions were obtained to demonstrate the temporal periods of limit cycles and Hopf bifurcation. It was further verified that under certain conditions on the radius of a disk with respect to the reaction-diffusion rates, the instability in the dynamics is restricted to Turing-type only forbidding the existence of a region for temporal bifurcation. A distinction between the transcritical and Hopf bifurcations was numerically established by analysing the temporal periods between transitional instabilities in the dynamics from one spatial pattern to another. The methodology used in the current paper sets an exclusive framework for a strategy to investigate general reaction-diffusion systems on arbitrary geometries. The current work brings together two important and routinely used approaches namely linear stability theory and the numerical computation to obtain robust and complete insight on the parameter spaces and the influence and role of domain size on the bifurcation theory.  
\subsection{Ideas for future work}
 We are currently extending this theoretical framework to study bifurcation analysis for reaction-diffusion systems with cross-diffusion \cite{paper7, paper20}. We want to study or investigate whether the existence of cross-diffusion has any influence on the conditions derived for the domain size in the context of bifurcation analysis with independent diffusion rates. Using the methodology of this work, arbitrary domain geometries may be investigated to obtain a full insight on the parameter spaces and possible relationship between reaction-diffusion rates with the domain size in the context of linear stability theory. 
 
 The strategy used in this paper can also be applied to study bulk-surface reaction-diffusion systems on three dimensional domains \cite{paper3, paper4, paper17, paper19}. It can be investigated to find how the surface area and the volume of three dimensional domains influence the spatio-temporal behaviour of a bulk-surface reaction-diffusion system. 
 
 Furthermore, our theoretical and conceptual framework allows us to investigate how continuous domain and surface growth and evolution influences the bifurcation behaviour of the system \cite{paper2, paper6, paper9, paper10, paper12, thesis1}. The strategy of the current work motivates us to investigate whether conditions (\ref{comp4}) and (\ref{condtur}) continue to influence the dynamics in the presence of  continuous domain and surface growth and evolution, or whether beyond a certain threshold of the domain or surface size, these conditions can be invalidated. Here, analytical theory on non-autonomous partial differential equations must be exploited appropriately.    
\section{Acknowledgments}
WS acknowledges support of the School of Mathematical and Physical Sciences Doctoral Training studentship. AM acknowledges support from the Leverhulme Trust Research Project Grant (RPG-2014-149) and the European Union's Horizon 2020 research and innovation programme under the Marie Sklodowska-Curie grant agreement No 642866. AM's work was partially supported by the Engineering and Physical Sciences Research Council, UK grant (EP/J016780/1). The authors (WS, AM) thank the Isaac Newton Institute for Mathematical Sciences for its hospitality during the programme (Coupling Geometric PDEs with Physics for Cell Morphology, Motility and Pattern Formation; EPSRC EP/K032208/1).  AM was partially supported by a fellowship from the Simons Foundation. AM is a Royal Society Wolfson Research Merit Award Holder, generously supported by the Wolfson Foundation.

\label{lastpage}

\end{document}